\documentclass[12pt,reqno]{amsart}
\usepackage{amssymb}
\usepackage{amsfonts}
\usepackage{amsmath}

\newtheorem{theorem}{Theorem}[section]
\newtheorem{lemma}[theorem]{Lemma}
\newtheorem{cor}[theorem]{Corollary}
\newtheorem{prop}[theorem]{Proposition}
\theoremstyle{definition}
\newtheorem{algorithm}[theorem]{Algorithm}
\newtheorem{examples}[theorem]{Examples}
\newtheorem{example}[theorem]{Example}
\newtheorem{notn}[theorem]{Notation}

\newtheorem{defn}[theorem]{Definition}
\newtheorem{defns}[theorem]{Definitions}
\newtheorem{rmks}[theorem]{Remarks}
\newtheorem{rmk}[theorem]{Remark}

\DeclareMathOperator{\id}{id}

\DeclareMathOperator{\Der}{Der}
\DeclareMathOperator{\InnDer}{InnDer}
\DeclareMathOperator{\ch}{char}
\DeclareMathOperator{\Fract}{Fract}
\DeclareMathOperator{\aut}{Aut}
\DeclareMathOperator{\Endo}{End}
\DeclareMathOperator{\PBW}{PBW}

\newcommand{\ov}{\overline}
\newcommand{\Z}{{\mathbb Z}}
\newcommand{\C}{{\mathbb C}}
\newcommand{\N}{{\mathbb N}}

\newcommand{\K}{{\mathbb K}}

\newcommand{\x}{\mathbf x}
\newcommand{\y}{\mathbf y}
\newcommand{\cb}{\mathbf c}
\newcommand{\db}{\mathbf d}
\newcommand{\eb}{\mathbf e}

\newcommand{\fb}{\mathbf f}

\newcommand{\qnum}[2]{[#1]_{#2}}

\begin{document}
\title{Skew derivations of quantum tori and quantum affine spaces}

\author{David A. Jordan}
\address{School of Mathematics and Statistics\\
University of Sheffield\\
Hicks Building\\
Sheffield S3~7RH\\
UK}
\email{d.a.jordan@sheffield.ac.uk}

\begin{abstract} 
We determine the $\sigma$-derivations of quantum tori and quantum affine spaces for a toric automorphism $\sigma$.
By standard results, every toric automorphism $\sigma$ of a quantum affine space $\mathcal{A}$ and every $\sigma$-derivation of $\mathcal{A}$
extend uniquely to the corresponding quantum torus $\mathcal{T}$. 
We shall see that, for a toric automorphism $\sigma$, every  $\sigma$-derivation of $\mathcal{T}$ is a unique sum of an inner $\sigma$-derivation and a $\sigma$-derivation that is conjugate to a derivation and that the latter is non-zero only if $\sigma$ is an inner automorphism of $\mathcal{T}$. This is applied to determine the $\sigma$-derivations of $\mathcal{A}$ for a toric automorphism $\sigma$, generalizing  
results of Alev and Chamarie for the derivations of quantum affine spaces and of Almulhem and Brzezi\'{n}ski for $\sigma$-derivations of the quantum plane. 

We apply the results to iterated Ore extensions $A$ of 
the base field for which all the defining endomorphisms are automorphisms and each of the adjoined indeterminates is an eigenvector for all the subsequent defining automorphisms. We present an algorithm which,  in characteristic zero, will, for such an algebra $A$, either construct a quantum torus between $A$ and its quotient division algebra or show that no such quantum torus exists.

Also included is a general section on skew derivations which become inner on localization at the powers of a normal element which is an eigenvector for the relevant automorphism. This section explores a connection between such skew derivations and normalizing sequences of length two. This connection is illustrated  by known examples of skew derivations and by the construction of a family of skew derivations for the parametric family of subalgebras of the Weyl algebra that has been studied in three papers by Benkart, Lopes and Ondrus.  
\end{abstract}

\maketitle

\section{Introduction}
Let $n$ be a positive integer and let $Q = (q_{ij})$ be a multiplicatively antisymmetric  $n \times n$  matrix
over a field $\K$, that is $q_{ii}=1$ for $1\leq i\leq n$ and,  for $1\leq i,j\leq n$, $q_{ij}\neq 0$ and $q_{ji}=q_{ij}^{-1}$. 
The \emph{quantum affine space}, $\mathcal{O}_Q(\K^n)$, is the $\K$-algebra generated by $x_1,x_2,\dots, x_n$
subject to the relations $x_ix_j = q_{ij}x_jx_i$ for $1 \leq i < j \leq n$.
The  \emph{quantum torus}, $\mathcal{O}_Q((\K^*)^n)$, is the $\K$-algebra generated by 
$x_1^{\pm1},x_2^{\pm1},\dots, x_n^{\pm1}$ subject to the relations $x_ix_j = q_{ij}x_jx_i$,  $1 \leq i < j \leq n$, and  $x_ix_i^{-1}=1=x_i^{-1}x_i$,  $1 \leq i \leq n$.

The single-parameter versions of quantum affine space and the quantum torus are obtained
by choosing $q\in\K^*$ and taking $q_{ii}=1$, for $1\leq i\leq n$, and $q_{ij} = q$ and $q_{ji}=q^{-1}$, for $1\leq i < j\leq n$. These will be denoted $\mathcal{O}_q(\K^n)$ and
$\mathcal{O}_q((\K^*)^n)$ respectively. When $n=2$,  the single-parameter and multiparameter versions
coincide, with $Q=\left(\begin{smallmatrix}
1&q\\
q^{-1}&1\\
\end{smallmatrix}
\right)$, and we shall  write the generators as $x$ and $y$, rather than $x_1$ and $x_2$, and refer to $\mathcal{O}_q(\K^2)$ as the \emph{ quantum plane}. 

We shall also consider certain intermediate algebras between $\mathcal{O}_Q(\K^n)$ and $\mathcal{O}_Q((\K^*)^n)$. Let $k$ be an integer such that $0\leq k\leq n$ and select integers
$i_1, i_2\dots i_k$ such that $i_j < i_{j+1}$ for $1 \leq j \leq k-1$. The $\K$-subalgebra $\mathcal{L}$ of $\mathcal{O}_Q((\K^*)^n)$ generated by 
$x_1,x_2,\dots, x_n,x_{i_1}^{-1},x_{i_2}^{-1},\dots, x_{i_k}^{-1}$
 will be denoted
$\mathcal{O}_Q(\K^n)_{\langle i_1,i_2,\dots,i_k\rangle}$ and referred to as a \emph{selectively localized quantum space}. We may refer to $x_1,x_2,\dots, x_n,x_{i_1}^{-1},x_{i_2}^{-1},\dots, x_{i_k}^{-1}$ as the \textit{canonical generators} of $\mathcal{L}$.

For 
$\Lambda=(\lambda_1, \lambda_2,\dots,\lambda_n)\in (\K^*)^n$, any selectively localized quantum space, including  $\mathcal{O}_Q(\K^n)$ and
$\mathcal{O}_Q((\K^*)^n)$, has a \textit{toric} automorphism $\sigma_\Lambda$  such that, for $1\leq i\leq n$, 
$\sigma_\Lambda(x_i)=\lambda_ix_i$.

Although the derivations of $\mathcal{O}_Q(\K^n)$ are well understood through the work of Alev and Chamarie \cite{alevchamarie}, little appears
to be known about the skew derivations of $\mathcal{O}_Q(\K^n)$. The only result of which we aware is \cite[Theorem 6.2]{AlBr}, where Almulhem and Brzezi\'{n}ski determine 
the right $\sigma$-derivations of $\mathcal{O}_q(\K^2)$ for the toric automorphism 
$\sigma$ of $\mathcal{O}_q(\K^2)$ such that $\sigma(x) = qx$ and $\sigma(y) = q^{-1}y$. The left $\sigma$-derivations may be deduced by symmetry. 
In \cite{AlBr}, the motivation  comes from differential geometry and algebraic properties of the resulting Ore extensions are not addressed.

We shall determine the $\sigma_\Lambda$-derivations of an arbitrary selectively localized quantum space for any toric automorphism $\sigma_\Lambda$, beginning with the 
case of the quantum torus $\mathcal{T}=\mathcal{O}_Q((\K^*)^n)$ for  which we exploit the $\Z^n$-grading for which, for $\db=(d_1,d_2,\dots,d_n)\in \Z^n$, the $\db$-component 
$\mathcal{T}_\db$ is one-dimensional spanned by $\x^{\db}=x_1^{d_1}x_2^{d_2}\dots x_n^{d_n}$. It will suffice to determine, for each $\db\in \Z^n$, the homogeneous $\sigma_{\Lambda}$-derivations of $\mathcal{T}$ of weight $\db$, that is those $\sigma_{\Lambda}$-derivations $\delta$ such that $\delta(\mathcal{T}_\cb)\subseteq \mathcal{T}_{\cb+\db}$ for all $\cb\in \Z^n$. Every $\sigma_{\Lambda}$-derivation of $\mathcal{T}$ is a unique sum of homogeneous $\sigma_{\Lambda}$-derivations. There is a striking dichotomy for the homogeneous $\sigma_{\Lambda}$-derivations. Given a homogeneous $\sigma_{\Lambda}$-derivation $\delta$  of $\mathcal{T}$ of weight $\db$, either the automorphism $\sigma_\Lambda$ is inner on $\mathcal{T}$, induced by $\x^{-\db}$, or the $\sigma_{\Lambda}$-derivation $\delta$ is inner on $\mathcal{T}$. 
In the latter case the space of homogeneous $\sigma_{\Lambda}$-derivations of $\mathcal{T}$ of weight $\db$ is $1$-dimensional while in the former case it is $n$-dimensional.
It follows, by standard results, that every Ore extension $\mathcal{T}[z;\sigma_\Lambda,\partial]$ is isomorphic to an Ore extension of either  derivation type or automorphism type. 

By standard localization results, every toric automorphism $\sigma_\Lambda$ and every $\sigma_\Lambda$-derivation of quantum space $\mathcal{A}=\mathcal{O}_Q(\K^n)$ extend to the quantum torus $\mathcal{T}=\mathcal{O}_Q((\K^*)^n)$ so the $\sigma_\Lambda$-derivations of $\mathcal{A}$ are the restrictions to $\mathcal{A}$ of those $\sigma_\Lambda$-derivations $\delta$ of $\mathcal{T}$  such that $\delta(\mathcal{A})\subseteq \mathcal{A}$. Any homogeneous $\sigma_\Lambda$-derivation of $\mathcal{T}$ of weight $\db=(d_1,d_2,\dots,d_n)$, where each $d_i\geq 0$, restricts to a 
$\sigma$-derivation of $\mathcal{A}$. We shall see that the only other weights $\db$ for which  there may be non-zero homogeneous $\sigma_\Lambda$-derivations $\delta$ of $\mathcal{A}$ of weight $\db$ are such that, for some $j$ with $1\leq j\leq n$, $d_j=-1$ and $d_i\geq 0$ if $i\neq j$. For such $\db$ we shall see that the space of homogeneous $\sigma_{\Lambda}$-derivations $\delta$ of $\mathcal{A}$ of weight $\db$ has dimension $0$ or $1$ and determine the conditions on the parameters $q_{ij}$ and $\lambda_i$ for the dimension to be $1$. These $\sigma_{\Lambda}$-derivations may fall on either side of the dichotomy that we observed in $\mathcal{T}$: either $\delta$ is outer on $\mathcal{A}$ but inner on the localization $\mathcal{A}_{\langle j\rangle}$ of $\mathcal{A}$ at the powers of $x_j$ or $\sigma_\Lambda$ is outer on $\mathcal{A}$ but inner on $\mathcal{A}_{\langle j\rangle}$.   The situation is similar for arbitrary selectively localized quantum spaces.

When, for an automorphism $\sigma$ of a $\K$-algebra $R$,  a $\sigma$-derivation $\delta$ becomes an inner $\sigma$-derivation on localization at the powers of a normal element $s$, the Ore extension $R_s[x;\sigma,\delta]$ of the localization $R_s$ becomes, by a standard change of variable, an Ore extension $R_s[x^\prime;\sigma]$ of automorphism type.  Following the influential work of Cauchon \cite{cauchon1} on quantum matrices, this process has become known as \textit{deleting derivations}.  We shall discuss this process in the general context of a $\K$-algebra $R$ with automorphisms $\sigma$, $\phi$ and $\gamma$
and a regular normal element $s$  such that $\sigma(s)=\nu s$ for some $\nu\in \K^*$, $rs=s\phi(r)$ for all $r\in R$ and $\gamma=\phi^{-1}\sigma$. Here, the automorphisms $\sigma$, $\phi$ and $\gamma$ extend to automorphisms of the localization  $R\mathcal{S}^{-1}$ and they induce automorphisms $\ov{\sigma}$, 
$\ov{\phi}$ and $\ov{\gamma}$ of $\ov{R}:=R/sR$. For $t\in R$, we shall say that a $\sigma$-derivation $\delta$ is $s$-\emph{locally inner}, \emph{induced by} $s^{-1}t$, if the extension of $\delta$ to the localization $R\mathcal{S}^{-1}$ is  inner, induced by $s^{-1}t$. The element $sx-t$ of the Ore extension $R[x;\sigma,\delta]$ is then normal in the Ore extension $R_s[x;\sigma,\delta]$ and we will give a sufficient condition for $sx-t$ to also be normal in $R[x;\sigma,\delta]$. We shall observe a connection between $s$-locally inner $\sigma$-derivations and normalizing sequences of length $2$ that we have not found explicitly in the literature. There is an $s$-locally inner $\sigma$-derivation of $R$ induced by $s^{-1}t$ if and only if $(s,t)$ is a normalizing sequence with $rt-t\gamma^{-1}(r)\in sR$ for all $r\in R$. This observation can facilitate the analysis of skew derivations in algebras with distinguished normal elements and we discuss several illustrative  examples. Many of these are well-known skew derivations but they include an apparently new class of locally inner skew derivations for the algebras studied  by Benkart, Lopes and Ondrus in \cite{BLOI,BLOII,BLOIII} for which the prototype is the Jordan plane.

 Let $R$ be an iterated Ore extension of the form 
$\K[x_1][x_2;\sigma_2,\delta_2]\dots[x_n;\sigma_n,\delta_n]$,
where, for $2\leq i \leq n$,  $\sigma_{i}$ and $\delta_i$ are, respectively, an automorphism and a $\sigma_{i}$-derivation of $R_{i-1}$, where, for $1\leq k\leq n$,  $R_k=\K[x_1][x_2;\sigma_2,\delta_2]\dots[x_{k};\sigma_{k},\delta_{k}]$, such that, for $1\leq j<i$,
there exists $\lambda_{ij}\in \K^*$ such that $\sigma_{i}(x_j)=\lambda_{ij}x_j$. This includes quantum nilpotent algebras \cite{GoodYakmem}, also known as $CGL$-extensions \cite{llrufd}, the definition of which has extra conditions 
including local nilpotency of the skew derivations. 
For quantum nilpotent algebras there are two algorithms, due to Cauchon \cite{cauchon1} and 
Goodearl and Yakimov \cite{GoodYakadv}, that, given a quantum nilpotent algebra $R$ with quotient division algebra $\Fract(R)$, will construct a quantum torus $\mathcal{T}$ such that $R\subseteq \mathcal{T}\subset \Fract(R)$. 
We shall apply our results to obtain an algorithm which, given an iterated Ore extension $R$ as above, will either yield a quantum torus $\mathcal{T}$ such that $R\subseteq \mathcal{T}\subset \Fract(R)$ or yield  a copy of  the  first Weyl algebra $A_1$ inside $\Fract(R)$. If $\ch \K=0$ then, by a result of Richard \cite{richard}, the quotient division algebra of a quantum torus $\mathcal{O}_Q((\K^*)^m)$ does not contain a copy of $A_1$ so the two outcomes are mutually exclusive.

Section 2 is a preliminary section presenting definitions and relevant standard results.  Locally inner skew derivations in general are discussed in Section 3.
Sections 4, 5 and 6 cover the determination of the $\sigma$-derivations, for a toric automorphism $\sigma$, of $\mathcal{O}_Q((\K^*)^n)$, $\mathcal{O}_Q(\K^n)$ and  $\mathcal{O}_Q(\K^n)_{\langle i_1,i_2,\dots,i_k\rangle}$ respectively. 
Section 7 is devoted to examples illustrating the results of those three sections and Section 8 presents the algorithm for iterated Ore extensions described above.

\section{Preliminaries}
The set of non-negative integers (resp. positive integers), will be denoted by $\N_0$ (resp. $\N$). Throughout $\K$ will denote a field and $\K\backslash{0}$ will be denoted by $\K^*$.  
 For a $\K$-algebra $R$ with identity,  the centre and the group of units of  $R$ will be denoted by $Z(R)$ and $U(R)$ respectively.  An \textit{automorphism} $\sigma$ of $R$ is a $\K$-linear ring automorphism of $R$, the $\K$-linearity ensuring that $\sigma|_\K=\id_\K$. The group of automorphisms of $R$ will be denoted by $\aut(R)$.

\subsection{Skew derivations}\label{endoskewder}
Let $R$  be a $\K$-algebra and let $\sigma\in\aut(R)$. A $\K$-linear map $\delta: R\rightarrow R$ is a
\emph{left} $\sigma$-\emph{derivation} of $R$ if 
\[\delta(rs) = \sigma(r)\delta(s) + \delta(r)s \text{ for all }r,s\in R\]
and is a \emph{right} 
$\sigma$-\emph{derivation} of $R$ if 
\[\delta(rs) = r\delta(s) + \delta(r)\sigma(s) \text{ for all }r,s\in R.\]
The $\K$-linearity ensures that $\delta(\K)=0$. When we refer to a $\sigma$-derivation without
specifying a side, we shall mean left $\sigma$-derivation. Any map $\delta:R \rightarrow  R$ that is a $\sigma$-derivation for some  $\sigma\in \aut(R)$ may be referred to as a \emph{skew derivation} of $R$.  If
$\sigma=\id_R$ then the definitions of left and right $\sigma$-derivations both reduce to the usual definition of a \emph{derivation} of $R$.

For $\sigma\in \aut(R)$, the set of all $\sigma$-derivations of $R$ will be denoted $\Der_\sigma(R)$. Clearly $\Der_\sigma(R)$ is closed under addition and if $z \in Z(R)$ then  $z\delta\in \Der_\sigma(R)$ for all $\delta\in \Der_\sigma(R)$. Thus $\Der_\sigma(R)$ is a $Z(R)$-submodule of $\Endo_{\K}(R)$. 

\begin{defns}
Let $R$ be a $\K$-algebra, let $\sigma\in \aut(R)$ and let $\delta\in \Der_\sigma(R)$.
\begin{enumerate}
\item For all $a\in R$, the map  $\delta_a:R\rightarrow R$ such that 
$\delta_a(r)=ar-\sigma(r)a$ for all $r\in R$ is a $\sigma$-derivation, called the \emph{inner} $\sigma$-\emph{derivation of} $R$ \emph{induced by} $a$. 
The set of all inner $\sigma$-derivations of $R$ is a $Z(R)$-submodule
of $\Der_\sigma(R)$ and will be denoted $\InnDer_\sigma(R)$. 
\item For $q\in \K^*$, $\delta$ is said to be $q$-\emph{skew} if
$\delta\sigma = q\sigma\delta$. Note that, for some authors, this would be the definition of a $q^{-1}$-\emph{skew} $\sigma$-derivation.
\item $\delta$ is said to be \emph{locally nilpotent} if, for all $r\in R$, there exists $n\in \N$ such that $\delta^n(r)=0$.
\end{enumerate}
\end{defns}

\begin{rmks}\label{qskewness} Let $R$ be a $\K$-algebra, let $\sigma\in \aut(R)$ and let $\delta\in \Der_\sigma(R)$. Let $G$ be a set of generators for $R$.
\begin{enumerate}
\item  It is easily checked that $\ker \delta$ is a
subalgebra of $R$ so if $G\subseteq \ker \delta$ then $\delta = 0$. Applying this to $\delta-\gamma$, where  
$\delta,\gamma\in \Der_\sigma(R)$, we see that if $\delta(g)=\gamma(g)$ for all $g\in G$ then
$\delta=\gamma$.
\item  For $q\in \K^*$, the linear map $\theta=\delta\sigma-q\sigma\delta:R\rightarrow R$ is an example of a $(\sigma^2; \sigma)$-derivation
of $A$, meaning that 
$\theta(rs)=\sigma^2(r)\theta(s)+\theta(r)\sigma(s)$
  for all $r,s\in R$. Consequently, $\ker \theta$ is a subalgebra of $R$ and, to check
that $\delta$ is $q$-skew, it suffices to check that $G\subseteq \ker \theta$.
\end{enumerate}
\end{rmks}

\subsection{Regularity and normality}\label{regnorm} Let $R$ be a $\K$-algebra and let $s\in R$.
 If $rs\neq 0$ and $sr\neq 0$ whenever $0\neq r\in R$ then $s$ is \emph{regular}. A  left or right Ore set $\mathcal{S}$ in $R$
will be called \textit{regular} if every element of $\mathcal{S}$ is regular. 

If  $Rs=sR$ then $s$ is a \emph{normal} element of $R$. 
If $s$ is normal and regular then $\{s^i\}$ is a right and left regular Ore set in $R$. If $\phi\in \aut(R)$ is such that $rs = s\phi(r)$ for all $r\in R$, or, equivalently,  $sr = \phi^{-1}(r)s$ for all $r\in R$,
then $s$ is $\phi$-\textit{normal}. If $s$ is normal and regular  there is a unique automorphism $\phi_s$ of $R$, the \emph{normalizing automorphism}
of $R$ \emph{induced by} $s$,  such
that $s$ is $\phi_s$-\textit{normal}. If $s$ is central and regular  then 
$\phi_s = \id_R$.  
If $s$ is a unit then 
$\phi_s$ is the inner automorphism $a\mapsto s^{-1}as$, which we may refer to as the \emph{inner automorphism} of $R$ \emph{induced by} $s$.

Normalizing sequences of length two, in the sense of \cite[4.1.13]{McCR}, will be significant for us. Let $s, t\in R$. The sequence $(s,t)$ is \textit{normalizing} if $s$ is normal in $R$ and $\ov{t}:=t+sR$ is normal in $\ov{R}:=R/sR$. If, in addition, $s$ is regular in $R$ and $\ov{t}$ is regular in $\ov{R}$ then $(s,t)$ is  \textit{regular}. If  $\phi$ and $\psi$ are automorphisms of $R$ and $R/sR$ respectively such that $s$ is $\phi$-normal in $R$ and $\ov{t}$ is $\psi$-normal in $\ov{R}$ we shall say that $(s,t)$ is a $(\phi,\psi)$-\textit{normalizing sequence}, in which case if $s$ is regular in $R$ then $\phi(s)=s$ and if $\ov{t}$ is regular in $\ov{R}$ then $\psi(\ov{t})=\ov{t}$.

\subsection{Ore extensions}\label{Oreext} 
We shall only consider Ore extensions of a $\K$-algebra $R$  when  $\sigma\in \aut(R)$ and $\delta\in \Der_\sigma(R)$. In this situation, let  $S$ be a $\K$-algebra extension of $R$. We write $S = R[x; \sigma,\delta]$, and say that $S$ is a \emph{Ore 
extension} of $R$ if there exists  $x\in S$ such that $S$ is a free left $R$-module with basis $\{1,x,x^2,\dots\}$ and, for
all $r\in R$, $xr = \sigma(r)x + \delta(r)$.    Given $R$, $\sigma$  and $\delta$, an Ore
extension $R[x; \sigma,\delta]$ always exists, by \cite[Proposition 2.3]{GW}, and, by \cite[Corollary 2.5]{GW}, any two such extensions $R[x; \sigma,\delta]$ and $R[y; \sigma,\delta]$ are
isomorphic. 
If $\sigma=\id_R$, so that $xr=rx+\delta(r)$ for all $r\in R$, then we write $S=R[x;\delta]$ and refer to $S$ as an Ore extension of \emph{derivation type}. 
On the other hand, if $\delta=0$, so that $xr=\sigma(r)x$ for all $r\in R$, we write $S=R[x;\sigma]$
and refer to $S$ as an Ore extension of \emph{automorphism type}. If $\delta\neq 0$ and $\sigma\neq \id_R$ we refer to $S$ as an Ore extension of \emph{mixed type}. 

A presentation of an Ore extension $S=R[x;\sigma,\delta]$, and possibly its type, may be changed by replacing the indeterminate $x$ by $sx-a$ where $s\in U(R)$ and $a\in R$ and making appropriate adjustments to $\sigma$ and $\delta$. In the following proposition, (i) follows from the more general result \cite[Proposition 2]{wexler} and (ii) and (iii) are immediate from (i).

\begin{prop}\label{cog}
Let $R$ be a $\K$-algebra, let $\sigma\in\aut(R)$ and let $\delta\in\Der_\sigma(R)$. 
Let $a\in R$, let $s$ be a unit in $R$ and let $\gamma_s$ be the inner
automorphism $r\mapsto srs^{-1}$. Let $\sigma^\prime=\gamma_s\sigma$ and let $z=sx-a$. 
\begin{enumerate}
\item
$s\delta$ is a $\sigma^\prime$-derivation of $R$ and
$R[x;\sigma,\delta]=R[z;\sigma^\prime,\delta^\prime]$, where $\delta^\prime=s\delta-\delta_a$ and $\delta_a$ is the inner $\sigma^\prime$-derivation of $R$ induced by $a$. 
\item
If $\sigma$ is the inner automorphism of $R$ induced by $s^{-1}$ of $R$ then $s\delta$ is a derivation of $R$
and $R[x;\sigma,\delta]=R[sx;s\delta]$ is of derivation type.
\item
 If $\delta$ is the inner $\sigma$-derivation of $R$ induced by
$a$, then, taking $s=1$, $R[x;\sigma,\delta]= R[x-a; \sigma]$  is of automorphism type.
\end{enumerate} 
\end{prop}  
\begin{defn}
In the circumstances of Proposition~\ref{cog}(ii), we may say that the $\sigma$-derivation $\delta$ is \textit{conjugate} to the derivation $s\delta$.
\end {defn}

The definition of Ore extension leads to a well-behaved notion
of \emph{degree} and \emph{leading coefficients}, see \cite[p36 Definition and Exercise 2O]{GW}. 
Let $R$ be a $\K$-algebra, let $\sigma\in\aut(R)$, let $\delta\in\Der_\sigma(R)$ and let $S=R[x; \sigma,\delta]$. 
Easy degree arguments show that if $R$ is a domain then $S$ is a domain and $U(S)=U(R)$. 
By the general skew Hilbert Basis Theorem \cite[Theorem 2.6]{GW},  if $R$ is left Noetherian then $S$ is left Noetherian and similarly on the right.

In an Ore extension $S=R[x;\sigma]$ of automorphism type, $x$ is normal and regular so  $\mathcal{X}=\{x^i:i\geq 0\}$ is a regular right and left Ore set.  The localization $R\mathcal{X}^{-1}$ will be written $R[x^{\pm1};\sigma]$ and called a \emph{skew Laurent polynomial algebra}. As a left $R$-module, $R[x^{\pm1};\sigma]$ is free with basis $\dots, x^{-2},x^{-1},1, x,x^{2},\dots$ and, for $r\in R$ and $m\in \Z$,
$x^mr=\sigma^m(r)x^m$. If $R$ is a domain then $U(R[x^{\pm1};\sigma])$ is generated by $U(R)$ and $x$.

\subsection{Quantum spaces and quantum tori}
The quantum space $\mathcal{O}_Q(\K^n)$ is an iterated Ore extension of automorphism type of the
base field $\K$, with the intermediate algebras being of the form 
$\mathcal{O}_{Q_{k+1}}(\K^{k+1})=\mathcal{O}_{Q_k}(\K^k)[x_{k+1}; \psi_k]$,
where $Q_k$ and $Q_{k+1}$ are obtained from $Q$ by deleting the appropriate rows and columns of $Q$ and, for $1 \leq i \leq k$,  $\psi_k(x_i) = q_{k,i}x_i$. The canonical generators $x_1,x_2,\dots, x_n$
may be adjoined in any order and, for each permutation $\sigma \in S_n$, $\mathcal{O}_Q(\K^n)$ has a $\PBW$ basis
$\{x_{\sigma(1)}^{a_1}x_{\sigma(2)}^{a_2}\dots x_{\sigma(n)}^{a_n}:a_i\geq 0\}$.

The quantum torus $\mathcal{O}_Q((\K^*)^n)$ is  an iterated skew Laurent polynomial
algebra over $\K$ and may be obtained from $\mathcal{O}_Q(\K^n)$ by localizing at the powers of the normal
element $x_1x_2\dots x_n$. These algebras are sometimes called \emph{McConnell-Pettit algebras} in recognition
of the paper \cite{McCP} wherein the following simplicity criterion appears as Proposition 1.3.
\begin{prop}\label{McCPcrit} 
The following are equivalent:
\begin{enumerate}
\item if $m_1, m_2,\dots,m_n\in\Z$ are such that $q_{1j}^{m_1}
q_{2j}^{m_2}\dots
q_{nj}^{m_n} = 1$ whenever $1 \leq j \leq n$ then
$m_1 = m_2 =\dots= m_n = 0$;
\item
$ Z(\mathcal{O}_Q((\K^*)^n)) = \K$;
\item $\mathcal{O}_Q((\K^*)^n)$ is simple.
\end{enumerate}
\end{prop}

\begin{rmks}\label{McCPrmks} 
(i) The equivalence of (i) and (ii) in the Proposition is due to the observa-
tion that $x_1^{m_1}
x_2^{m_2}\dots
x_n^{m_n}\in Z(\mathcal{O}_Q((\K^*)^n))$ if and only if $q_{1j}^{m_1}
q_{2j}^{m_2}\dots
q_{nj}^{m_n} = 1$ for $1\leq  j \leq n$.

(ii) Suppose that the equivalent conditions of the Proposition hold. Then the normal elements in
quantum space $\mathcal{O}_Q(\K^n)$ are the elements of the form $\lambda x_1^{m_1}
x_2^{m_2}\dots
x_n^{m_n}$, where $\lambda\in \K^*$ and
each $m_i\geq 0$, and the height one prime ideals $\mathcal{O}_Q(\K^n)$ are the $n$ ideals $x_i\mathcal{O}_Q(\K^n)$, $1 \leq i \leq n$. Consequently,
every automorphism 
of $\mathcal{O}_Q(K^n)$ is the composition of a toric automorphism and an automorphism $\alpha$ such that $\alpha(x_i)=x_{\tau(i)}$ for some $\tau\in S_n$ such that 
$q_{ij}=q_{\tau(i)\tau(j)}$ whenever $1\leq i,j\leq n$. 

It is often the case that the only such permutation $\tau$ is the identity but, as observed in \cite{alevchamarie}, in the case where $n=2m$ is even and 
$\mathcal{O}_Q(K^n)$ is the tensor product of $m$ copies of the quantum plane $\mathcal{O}_q(K^2)$, every non-trivial permutation in $S_m$ gives rise to a non-toric automorphism of $\mathcal{O}_Q(K^n)$.

(iii) It is shown in \cite{alevchamarie} that if $n\geq 2$, $n\neq 3$ and $q$ is not a root of unity, every automorphism of the single-parameter quantum space $\mathcal{O}_q(K^n)$ is toric. When $n=3$, there is a non-toric automorphism of $\mathcal{O}_q((\K^*)^3)$ under which $x_1\mapsto x_1$, $x_2\mapsto x_1x_3+x_2$ and $x_3\mapsto x_3$.

(iv) A selectively localized quantum space
$\mathcal{L}=\mathcal{O}_Q(\K^n)_{\langle i_1,i_2,\dots,i_k\rangle}$ is the localization
of $\mathcal{O}_Q(\K^n)$ at the powers of the normal element $x_{i_1}x_{i_2}\dots x_{i_k}$ and may be obtained from $\K$ in
$n$ steps, each of which involves taking either an Ore extension of automorphism type
or a skew Laurent polynomial algebra.  If $k>0$ then, in stating results on selectively localized quantum spaces,
we may, by replacing $Q$ by its conjugate by the appropriate permutation matrix, assume that
$i_j = j$ for $1 \leq j \leq k$. 
By Subsection~\ref{Oreext}, $\mathcal{L}$ is a right and
left Noetherian domain and its group of units is generated by $\K^*$ and the elements $x_{i_j}$, $1\leq j\leq k$.
\end{rmks}

\section{Locally inner skew derivations}

\subsection{Deleting derivations}\label{deldersitu}
Following the influential work of Cauchon \cite{cauchon1} on quantum matrices, the phrase
\emph{deleting derivations} has been used to cover the situation where a skew derivation becomes inner
in a localization so that, in the passage to the Ore extension over the relevant localization,
an Ore extension of mixed 
type becomes, by Proposition~\ref{cog}(iii), an Ore extension
of automorphism type. 
The idea of deleting derivations is applicable in the following situation which will be in place throughout this section: $R$ is a $\K$-algebra with an automorphism $\sigma$
and a regular normal element $s$  such that $\sigma(s)=\nu s$ for some $\nu\in \K^*$.  
The normalizing automorphism $\phi_s$ will be denoted by $\phi$ and $\phi^{-1}\sigma$ will be denoted by $\gamma$. The factor $R/sR$ will be denoted by $\ov{R}$
and $\pi:R\rightarrow \ov{R}$ will be the canonical projection. As $\sigma(sR)=\phi(sR)=\gamma(sR)=sR$, there are induced automorphisms $\ov{\sigma}$, 
$\ov{\phi}$ and $\ov{\gamma}$ of $\ov{R}$. Also  the right and left Ore set $\mathcal{S}:=\{\mu s^j : \mu\in \K^*, j\in \N\}$ is invariant under $\sigma$, $\phi$ and 
$\gamma$ which therefore extend to automorphisms of the localization  $R\mathcal{S}^{-1}$, 
with $\psi(ab^{-1})=\psi(a)\psi(b)^{-1}$ for $a\in R$ and $b\in \mathcal{S}$, where $\psi=\sigma, \phi$ or $\gamma$.
 
Let $\delta$ be a $\sigma$-derivation of $R$ and let $S = R[x; \sigma, \delta]$. By \cite[Lemmas 1.3 and 1.4]{gprimeskr},
 $\delta$ extends uniquely to the localization $R\mathcal{S}^{-1}$, $\mathcal{S}$ is a
regular right Ore set in $S$ and 
 $T:=
R\mathcal{S}^{-1}[x; \sigma, \delta]$ is a right quotient algebra of $S$ with respect to $\mathcal{S}$.

\begin{defn}\label{locally inner} Let $t\in R$ and note that $s^{-1}t=\phi(t)s^{-1}\in 
R\mathcal{S}^{-1}$. \begin{enumerate}
\item We say that a $\sigma$-derivation $\delta$ of $R$ is $s$-\emph{locally inner}, \emph{induced by} $s^{-1}t$, if the extension of $\delta$ to the localization $R\mathcal{S}^{-1}$ is  inner, induced by
 $s^{-1}t$.  If $\sigma=\id_R$ we may refer to $\delta$ as an $s$-\emph{locally inner derivation}.
\item If the extension of $\sigma$ to $R\mathcal{S}^{-1}$ is inner then the extension of $\delta$ to $R\mathcal{S}^{-1}$ is conjugate to a derivation and we may say that $\delta$ is $s$-\emph{locally conjugate to a derivation}.
\end{enumerate}
\end{defn}
\begin{notn}\label{T}
The $s$-locally inner $\sigma$-derivations of $R$ form a subspace of $\Der_\sigma(R)$ that will be denoted $\InnDer_{\sigma,s}(R)$ and the set 
$\{t\in R: \text{there exists }\delta\in \InnDer_{\sigma,s}(R)\text{ induced by }s^{-1}t\}$ will be denoted $T(R,\sigma,s)$. For $r\in R$, the inner $\sigma$-derivation
of $R$ induced by $r$ is $s$-locally inner induced by $s^{-1}sr$ so $\InnDer_{\sigma}(R)\subseteq \InnDer_{\sigma,s}(R)$ and $sR\subseteq T(R,\sigma,s)$.
\end{notn}

\begin{prop}\label{delder1}
For $t\in R$,
 $t\in T(R,\sigma,s)$ if and only if $(s,t)$ is a $(\phi,\ov{\gamma}^{-1})$-normalizing sequence.
\end{prop}
\begin{proof}
Let $\delta$ be the inner $\sigma$-derivation of $R\mathcal{S}^{-1}$  induced by
 $s^{-1}t$. For $r\in R$,
\[\delta(r)=s^{-1}tr-\sigma(r)s^{-1}t=s^{-1}(tr-\gamma(r)t).\]
Thus $t\in T(R,\sigma,s)$ if and only if $tr-\gamma(r)t \in sR$ for all $r\in R$, that is, if and only if $(s,t)$ is $(\phi,\ov{\gamma}^{-1})$-normalizing. 
\end{proof}

\begin{cor}\label{deldercor}
If $\ov{\gamma}=\id_{\ov{R}}$ then $T(R,\sigma,s)=\pi^{-1}(Z(\ov{R}))$.
\end{cor}
\begin{proof}
This is immediate from Proposition~\ref{delder1}.
\end{proof}

\begin{prop}\label{delder2} Let $t\in T(R,\sigma,s)$ and let $\delta$ be the $s$-locally inner $\sigma$-derivation $\delta$   
induced by $s^{-1}t$.
\begin{enumerate}
	\item  For all $r\in R$, $\delta(r)=(\phi(t)\phi(r)-\sigma(r)\phi(t))s^{-1}=s^{-1}(tr-\gamma(r)t)$.
\item Let  
$T=R\mathcal{S}^{-1}[x; \sigma, \delta]$. 
Then $T=
(R\mathcal{S}^{-1})[w;\gamma]$,
where  $w=sx-t$ and $\gamma$ is extended to an automorphism of $R\mathcal{S}^{-1}$. 
\item   Suppose that $\phi(t)\equiv \sigma(t)\bmod sR$ and let $b=\sigma^{-1}(s^{-1}(\phi(t)-{\sigma}(t))$.
The element $w$ is normal in $S$ with $wr=\gamma(r)w$ for all $r\in R$, $xw=w(\nu x+b)$ and $wx=\nu^{-1}(x-\gamma(b))w$. 
In particular, if $\phi(t)=\sigma(t)$ then  $xw=\nu wx$ and if, further, $\nu=1$ and $\gamma|_R=\id|_R$  then $w$ is central in $S$.
\item If the $(\sigma,\ov{\gamma}^{-1})$-normalizing sequence $(s,t)$ is regular then $\phi(t)\equiv \sigma(t)\bmod sR$ so (iii) applies.

\end{enumerate}
\end{prop}
\begin{proof}

(i) This is clear from the proof of Proposition~\ref{delder1}.

(ii) By (i), $s\delta$ is the inner $\gamma$-derivation induced by $t$ so, $s$ being a unit in $R\mathcal{S}^{-1}$, this follows from Proposition~\ref{cog}(i).

(iii)  Note that $s\sigma(b)=\phi(t)-\sigma(t)$. By (i),
\begin{equation*}
\delta(t)=\gamma(b)t,\;\delta(s)=\phi(t)-\nu t \text{ and }
s\delta(b)=tb-\gamma(b)t.
\end{equation*}
It follows that
\begin{equation*}\label{xw}
xw=x(sx-t)=\nu sx^2+(\phi(t)-\nu t-\sigma(t))x-\gamma(b)t,
\end{equation*}
whereas, as $s\sigma(b)=\phi(t)-\sigma(t)$ and $s\delta(b)=tb-\gamma(b)t$,
 \begin{equation*}\label{wx}
w(\nu x+b)=(sx-t)(\nu x+b)=\nu sx^2+(\phi(t)-\sigma(t))x+(tb-\gamma(b)t)-\nu tx-tb=xw.
\end{equation*}
Thus $xw=w(\nu x+b)\in wS$ and $wx=\nu^{-1}(x-\gamma(b))w\in Sw$. As $wR=\gamma(R)w=Rw$ and $S$ is generated by $x$ and $R$, it follows that $w$ is normal in $S$.

If $\phi(t)=\sigma(t)$ then $b=0$ so $xw=\nu wx$. If, further, $\nu=1$ and $\gamma|_R=\id|_R$  then $wx=xw$ and $wr=rw$ for all $r\in R$ so $w\in Z(S)$.

(iv) Suppose that the $(\sigma,\ov{\gamma}^{-1})$-normalizing sequence $(s,t)$ is regular. As $\ov{t}$ is $\ov{\gamma}^{-1}$-normal, 
$t^2\equiv {\sigma}^{-1}{\phi}(t)t\bmod sR$ and, as $t$ is regular modulo $sR$,  
$t-\sigma^{-1}\phi(t)\in sR$. Applying $\sigma$, we see that ${\sigma}(t)-\phi(t)\in sR$.

\end{proof}

Propositions~\ref{delder1} and \ref{delder2} show how a normal element of degree $1$ in an  Ore extension may arise from locally inner skew derivations. The following result from 
\cite{normalOre} is, in a sense,  converse to them.

\begin{prop}\label{normalOreprop}
Let $\sigma$ be an automorphism of a domain $R$, let $\delta$ be a $\sigma$-derivation of $R$ and let $S$ be the Ore extension $R[x;\sigma,\delta]$. Let $c$ be a normal element of $S$ of the form $sx+t$, where $s,t\in R$ and $s\neq 0$. Then $s$ is normal in $R$ and $\delta$ is the $s$-locally inner $\sigma$-derivation of $R$ induced by $s^{-1}t$. Furthermore if $t$ is regular modulo $sR$ then $S/Sc$ is a domain.
\end{prop}
\begin{proof}
See \cite[Proposition 1 and its proof]{normalOre}.
\end{proof}

\subsection{Examples of locally inner skew derivations}

\begin{example}\label{qdisc}
Let $R$ be the polynomial algebra $\K[y]$, let $q\in \K^*$ and let $\sigma$ be the automorphism of $R$ such that $\sigma(y) = qy$. 
Let $s=y$ which is regular and normal in $R$. Then $\phi=\id_\K$, $R/yR=\K$ and $\ov{\gamma}=\id$.
By Corollary~\ref{deldercor}, $T(R,\sigma,y)=\pi^{-1}(Z(\ov{R}))=R=yR+\K$ so $\InnDer_{\sigma,y}=\InnDer_{\sigma}\oplus \K\delta$,
where $\delta$ is the
$y$-locally inner $\sigma$-derivation $\delta$ of $R$ induced by $y^{-1}$. By Proposition~\ref{delder2}(i), $\delta(y)=y^{-1}y-qyy^{-1}=1-q\in R$.

The Ore extension $S = R[x;\sigma,\delta]$ is generated as an $\K$-algebra by $x$ and $y$ subject to the
relation $xy-qyx = 1 - q$. This relation is equivalent to $yx-q^{-1}xy = 1 - q^{-1}$, giving rise to symmetry between $q$ and $q^{-1}$ and between $x$ and $y$.  The algebra $S$ is called the \emph{quantum disc} in \cite{AlBr} and, when
$q \neq 1$, it is isomorphic to the \emph{quantized Weyl algebra} $A_1^q$,
in which the defining relation is
$xy-qyx = 1$, for example, see \cite{GW}. There is a well-known normal element, $w=yx-1=q^{-1}(xy-1)$ of $S$, with $wy=qyw$ and $xw=qwx$, and its
existence is a simple illustration of Proposition~\ref{delder2}(iv). 
\end{example}

\begin{examples}\label{qplane}
 Let $q\in \K^*$, let $\mathcal{A}$ be the quantum plane $\mathcal{O}_q(K^2)$ and suppose that $q$ is not a root of unity. Thus $\mathcal{A}$ is generated by
$x$ and $y$ subject to  the relation $xy = qyx$. Let $\sigma$ be the toric automorphism
of $\mathcal{A}$ such that $\sigma(x) = qx$ and $\sigma(y) = q^{-1}y$.

Taking  $s=x$, which is regular and normal in $\mathcal{A}$,  
$\phi(y)=q^{-1}y$, $\ov{\mathcal{A}}=\K[\ov{y}]$, and $\ov{\gamma}=\id_{\K[\ov{y}]}$. By Corollary~\ref{deldercor}, $T(\mathcal{A},\sigma,x)=\pi^{-1}(Z(\ov{\mathcal{A}}))=\mathcal{A}=x\mathcal{A}+\K[y]$ so, given $h(y)\in \K[y]$, there is an $x$-locally inner $\sigma$-derivation $\delta_{h(y)}$ of $\mathcal{A}$ induced by $x^{-1}h(y)$. For $i\geq 0$,  $\delta_{y^i}(x)=(q^{-i}-q)y^i$ and $\delta_{y^i}(y)=0$. It follows that, if $\theta_y:\K[y]\rightarrow \K[y]$ is the $\K$-linear map such that $\theta_y(y^i)=(q^{-i}-q)y^i$ for $i\geq 0$, $\delta_{h(y)}(x)=\theta_y(h(y))$ and $\delta_{h(y)}(y)=0$. 

By the symmetry between $x$ and $y$ and between $q$ and $q^{-1}$, $T(\mathcal{A},\sigma,y)=y\mathcal{A}+\K[x]$ and, given $j(x)\in \K[x]$, there is a $y$-locally inner $\sigma$-derivation $\delta^\prime_{j(x)}$ of $\mathcal{A}$, induced by $y^{-1}j(x)$, such that $\delta^\prime_{j(x)}(y)=\theta_x(j(x))$ and $\delta^\prime_{j(x)}(x)=0$, $\theta_x:\K[x]\rightarrow \K[x]$ being the $\K$-linear bijection such that $\theta_x(x^i)=(q^i-q^{-1})x^i$ for $i\geq 0$.

Given $g(y)\in \K[y]$ and $f(x)\in \K[x]$, if $h(y)=\theta_y^{-1}(y)$ and $j(x)=\theta_x^{-1}(x)$ the 
$\sigma$-derivation $\delta_{f,g}:=\delta_{h(y)}+\delta^\prime_{j(x)}$ of $\mathcal{A}$, which is $xy$-locally inner,  induced by $(xy)^{-1}(qyh(y) + xj(x))$, is such that $\delta_{f,g}(x)=g(y)$ and $\delta_{f,g}(y) = f(x)$. These left $\sigma$-derivations are the analogues of the right $\sigma$-derivations that were identified in \cite[Theorem 6.2]{AlBr}.

Fix $f(x)$ and $g(y)$ and let $\delta=\delta_{f,g}$. The defining relations for the Ore extension $S=\mathcal{A}[z;\sigma,\delta]$ are then
\[xy=qyx,\quad zx=qz+g(y),\quad zy=q^{-1}yz+f(x).\]
By Proposition~\ref{delder2}, $S$ has a central element 
$w:=xyz-(qyh(y)+xj(x))$. The algebras $S$ and $S/wS$ have the potential for further study. Here we restrict to two comments on prime ideals. Provided $g(y)$ and $f(x)$ are both non-zero, 
 $qyh(y)+xj(x)$ is regular modulo $xy\mathcal{A}$ so, by Proposition~\ref{normalOreprop}, $S/wS$ is a domain and $wS$ is a completely prime ideal of $S$. 
The constant terms $a$ and $b$ of $g(y)$ and $f(x)$ are significant as 
$xS+yS$ is a  proper ideal, with $S/(xS+yS) \simeq \K[z]$, if and  only if $a=0=b$.

We regard some special cases as worthy of particular comment.
If $f(x)=1-q^{-1}$ and $g(y)=1-q$ the defining relations for $S$ are
\[xy = qyx,\quad zx-qxz = 1-q,\quad zy-q^{-1}yz=1-q^{-1}.\]
When $q\neq 1$ and $\K=\C$, $S$ is the \textit{linear connected quantized Weyl algebra} $L_q^3$,
see \cite{fishjordan}.
The element $w=xyz-qy-x$ is central and, by \cite[Proposition 6.1]{fishjordan},
the quantum cluster algebra of the quiver with Dynkin diagram of type $A_2$ is isomorphic to  
$S$.

From the point of view of graded algebras, the interesting case is where 
$g(y)=\tau y^2$ and $f(x)=\rho x^2$ for some $\tau,\rho\in \K$ and the defining relations  are
\[xy=qyx,\quad zx=qxz+\tau y^2,\quad zy=q^{-1}yz+\rho x^2.\]  If $\K$ is  algebraically closed, $\tau\in \K^*$ and $\rho\in \K^*$ then, replacing $x$ and $y$ by appropraite scalar multiples, we may assume that $\rho=\tau=1$. In this case $S$ the central element
$w$ is $xyz-(q^{-2}-q)^{-1}y^3-(q^3-1)^{-1}x^3$. A particular example appears in \cite[Example 2.1.5]{NAG}, where it is used to illustrate the Diamond
Lemma.

The $\sigma$-derivation $\delta_{f,g}$ may or may not be locally nilpotent. To ease notation, let $\delta=\delta_{f,g}$. 
One case where $\delta$ is locally nilpotent  is where $f(x)\in\K$ and $g(y)\in \K$, in which case, 
 for $m,n\geq 1$, 
$\delta(y^m)\in \K y^{m-1}$, $\delta(x^n)\in \K x^{n-1}$ and $\delta(y^mx^n)\in \K y^{m-1}x^n+\K y^mx^{n-1}$, whence $\delta^{m+n+1}(y^mx^n)=0$. Another is where $g(y)=0$ or $f(x)=0$. For example, suppose that $g(y)=0$, so that $\delta(x)=0$. A simple induction shows that $\delta(y^i)\in y^{i-1}\K[x]$ for $i\geq 1$, from which it follows that $\delta^{m+1}(y^m)=0$ for $m\geq  0$ and hence that $\delta^{m+1}(x^ny^m)=0$ for $n,m\geq 0$. Therefore $\delta$ is locally nilpotent. 

The easiest example in which to see that $\delta$ is not always locally nilpotent is where $f(x)=x$ and  
$g(y)=y$, so that the defining relations in $\mathcal{A}[z;\sigma,\delta]$ are 
\[xy = qyx, \quad zx = qxz + y, \quad zy = q^{-1}yz + x.\]
Here $\delta^n(x)=y$ if $n$ is odd and $\delta^n(x)=x$ if $n$ is even.

The $\sigma$-derivation $\delta_{f,g}$ may or may not be $\mu$-skew for some $\mu\in \K$. Let $\delta=\delta_{f,g}$ and write $g(y)=\sum_{m=0}^k a_my^m$ and  $f(x)=\sum_{m=0}^n c_mx^m$, where $a_i,c_i\in \K$ for all $i$.
Note that $\delta\sigma(x)=qg(y)$, $\sigma\delta(x)=g(q^{-1}y)$, $\delta\sigma(y)=q^{-1}f(x)$ and $\sigma\delta(y)=f(qx)$. 
For $\mu\in \K$, $\delta$ is $\mu$-skew if and only if $\mu=q^{i+1}$ for every $i$ such that $a_i\neq 0$ and 
$\mu=q^{-(i+1)}$ for every $i$ such that $c_i\neq 0$. In particular, if $q$ is not a root of unity then, by Remark~\ref{qskewness}(ii), the only situations where $\delta$ is $\mu$-skew 
for some $\mu\in \K^*$ are where either $f(x)=0$ and $g(y)=ay^i$ for some $a\in \K^*$,
in which case $\delta$ is $q^{-(i+1)}$-skew, 
or where $f(x)=cx^i$ for some $c\in \K^*$ and $g(y)=0$ in which case $\delta$ is $q^{i+1}$-skew. 
\end{examples}

\begin{examples}\label{ambi}
Let $A$ be a $\K$-algebra with an automorphism $\sigma$ and a regular normal element $s$ whose normalizing automorphism $\phi$ commutes with $\sigma$. Let $\gamma=\sigma\phi^{-1}$. Let $R$ be the Ore extension 
$A[y;\phi^{-1}]$ and extend $\phi$, $\gamma$ and $\sigma$ to automorphisms of $R$ by setting $\phi(y)=y$ and $\gamma(y)=\sigma(y)=\rho y$ for some $\rho\in \K^*$. Let $s=y$ which is regular and $\phi$-normal in $R$. We may identify $R/yR$ with $A$ while $\ov{\phi}$, $\ov{\gamma}$ and $\ov{\sigma}$ may be identified with the original automorphisms $\phi$, $\gamma$ and $\sigma$ of $A$.  For $t\in R$, the sequence $(y,t)$ is $(\phi,\ov{\gamma^{-1}})$-normalizing if and only if $t\equiv u\bmod yR$ for some $\gamma^{-1}$-normal element $u$ in $A$. Let $u$ be such an element. The $y$-locally inner $\sigma$-derivation $\delta_u$ of $R$ induced by $y^{-1}u$ is such that, for $a\in A$, 
\[\delta(a)=y^{-1}(ua-\gamma(a)u)=0\text{ and }\delta(y)=y^{-1}(uy-\rho yu)=\phi(u)-\rho u,\] so the Ore extension $S=R[x;\sigma,\delta]$ is the extension of $R$ generated by $x$ and $y$ subject to the relations
\[ya=\phi^{-1}(a)y\text{ and }xa=\sigma(a)x\text{, for all }a\in A \text{, and }xy-\rho yx=\phi(u)-\rho u.\]  By Proposition~\ref{delder2}(iii,iv), 
the element $w=yx-u$ is normal in $S$ 
with $wy=\rho yw$, $xw=\rho wx$ and $wa=\gamma(a)$ for all $a\in A$.

The Ore extension $S$ is a \emph{conformal ambiskew polynomial algebra}. The theory of such algebras has been developed, though not from the perspective of normalizing sequences,  in a sequence of papers including \cite{itskew,downup,simambi}, where the element $w$ is referred to as the \textit{Casimir element}. 

The quantum disc in Example \ref{qdisc} fits into this picture, with $A=\K$,
$\phi(y) = qy$, $\rho = q$, $u = 1$ and $w = yx- 1$. Another example is the enveloping algebra $U(\mathfrak{sl}_2)$ which arises when
$A$ is the polynomial algebra $\K[h]$, $\phi(h)=h-2$, $\rho=1$ and $u=-\frac{1}{4}(h+1)^2$, whence $\gamma=\id_A$, $\sigma=\phi$ and $\delta(y)=h$.
This is probably the best-known Ore extension of mixed type. The generating relations
are \[hy=y(h-2),\quad xh=(h-2)x,\quad xy-yx=h,\]
though $x$ and $y$ are commonly written as $e$ and $f$ respectively. 
Thus $U(\mathfrak{sl}_2)=R[x;\phi,\delta]$, where $R=\K[h][y,\phi^{-1}]$,
$\phi(h)=h-2$, $\phi(y)=y$, and $\delta$ is the $y$-locally inner $\phi$-derivation of $\K[h][y;\phi^{-1}]$ induced by $y^{-1}u$. 
\end{examples}
\begin{example}\label{xNambi}
Here we mention an ambiskew polynomial algebra that will be significant in Example~\ref{double}.
 Take $A=\K[z,g^{\pm 1}]$, a Laurent polynomial algebra over a polynomial algebra.
Let $\sigma$ be the automorphism of $A$ such that $\sigma(g)=g$ and $\sigma(z)=z+1$ and let $\phi=\sigma$ so that $\gamma=id_A$. Let $\rho=1$ and let $u=z(1-g)$. Then the ambiskew polynomial algebra $S$ is generated by $g$, $g^{-1}$, $z$, $x$ and $y$ subject to the relations $gg^{-1}=1=g^{-1}g$ and $gz=zg$ together with
\[xg=gx,\quad yg=gy, \quad xz=zx+x,\quad yz=zy-y,\quad xy-yx=1-g.\]
The Casimir element $w=yx-z(1-g)$ is central and, if $\ch(\K)=0$, \cite[Lemma 1.6]{itskew} is applicable to show that $Z(S)=\K[g^{\pm 1},w]$.
\end{example}

\begin{examples}\label{Rh} 
Here we present a class of examples illustrating the role that locally inner skew derivations can play in the analysis of skew derivations of algebras with known normal elements. 
They also provide  examples of the situation of Proposition~\ref{delder2}(iii) in which $\sigma(t)\neq \phi(t)$ so that the element $b$ such that $xw=w(\nu x+b)$ for some $\nu\in\K$ is non-zero. 

Let $h\in \K[x]$, let $\partial$ be the derivation $hd/dx$ of $\K[x]$ and let $R_h$ be the Ore extension $\K[x][y;\partial]$. Thus $R_h$ is the $\K$-algebra generated by $x$ and $y$ subject to the
relation $yx=xy+h$. 
The algebras $R_h$, which are embeddable in the first Weyl algebra $A_1$, were studied by Benkart, Lopes and Ondrus in \cite{BLOI,BLOII,BLOIII}. The automorphisms and derivations of $R_h$ were
determined in \cite{BLOI} and \cite{BLOIII} respectively but less appears to be known about the skew derivations.  

We shall only consider the case where $0\neq h\in x\K[x]$. If $\K$ is algebraically
closed, the general case where $\deg h\geq 1$ reduces to this case on replacement of the generator $x$ by $x-\lambda$ for a zero $\lambda$ of $h$, see
\cite[Theorem 8.2(i)]{BLOI}. In this case, let $R=R_h$ and let $f\in \K[x]$ be such that $h=xf$. Then $x$ is normal in $R$ with normalizing automorphism  $\phi$ such that
$\phi(y)=y+f$ and $\phi(x)=x$. Given $p\in \K[x]$, there is an automorphism $\sigma_p$ of $R$ such that $\sigma_p(y)=y+p$ and
$\sigma_p(x)=x$, see \cite[Section 8.1]{BLOI}. Let $\sigma=\sigma_p$ and note that $\phi=\sigma_f$. This fits into the situation  of Subsection~\ref{deldersitu} with $s=x$, $\gamma=\sigma_{p-f}$ and $\nu=1$. We shall determine the space $\InnDer_{\sigma,x}(R)$ of $x$-locally inner $\sigma$-derivations of $R$.

If $p-f\notin x\K[x]$ then $\ov{\gamma}\neq \id_{\ov{R}}$ whereas $\ov{R}$ is commutative so $\ov{R}$ cannot have a non-zero $\ov{\gamma}$-normal element and, by Proposition~\ref{delder1},  
$\InnDer_{\sigma,x}(R)=\InnDer_{\sigma}(R)$.  So henceforth we shall assume that 
$p-f\in x\K[x]$, $p-f=xd$, say, and hence that 
$\ov{\gamma}=\id_{\K[\ov{y}]}$.

By Corollary~\ref{deldercor}, $T(R,\sigma,x)=\pi^{-1}(Z(\ov{R}))=R$ so $\InnDer_{\sigma,x}(R)=\InnDer_\sigma(R)\oplus D$, where $D$ is spanned by the $x$-locally inner $\sigma$-derivations $\delta_k$ induced by $x^{-1}y^k$, $k\geq 0$. 

By Proposition~\ref{delder2}(i)~and~\cite[Exercise 2K]{GW}, 
\[\delta_k(x)=x^{-1}(y^kx-xy^k)=x^{-1}\sum_{i=0}^{k-1}\left(\begin{smallmatrix}k\\i\end{smallmatrix}\right)\partial^{k-i}(x)y^{i},\]
whence
\[
\delta_k(x)=\sum_{i=0}^{k-1}\left(\begin{smallmatrix}k\\i\end{smallmatrix}\right)a_iy^{i}\in R,
\] where $a_i=x^{-1}\partial^{k-i}(x)=fd/dx(\partial^{k-i-1}(x))$.
By Proposition~\ref{delder2}(i),
\[\delta_k(y)=x^{-1}(y^{k+1}-(y+p-f)y^k)=
 -dy^k.\] 
 
By Proposition~\ref{delder2}(iii), the element $w:=xz-y^k$ is normal in the Ore extension $R[z;\sigma,\delta_k]$ with
\[wx = xw, \quad wy=(x + (p-f)y)w \text{ and } 
zw = w(z-b),\] where $b=x^{-1}(y^k-(y+f-p)^k)\in R$.
If $k=0$ then $b=0$ and if $k=1$ then $b=x^{-1}p=d$.
\end{examples}

The next two examples are special cases of the algebras $R_h$ in Examples~\ref{Rh}, namely the enveloping algebra of the $2$-dimensional non-abelian solvable Lie algebra, in Example~\ref{R1}, and the Jordan plane, in Example \ref{J}.

\begin{example}\label{R1}
In Example~\ref{Rh}, let $h=x$ so that $f=1$ and the defining relaton for $R=R_x$ is $yx=xy+x$ as in Examples~\ref{Rh}. If $p\not\equiv 1\bmod x\K[x]$ then 
$\InnDer_{\sigma_p,x}(R)=\InnDer_{\sigma_p}(R)$.  If $p=dx+1$, where $d\in \K[x]$, then $\InnDer_{\sigma_p,x}(R)=\InnDer_{\sigma_p}(R)\oplus D$, where $D$ is spanned by the $x$-locally inner $\sigma_p$-derivations
$\delta_k$ induced by $x^{-1}y^k$, $k\geq 0$. Here $a_i=1$ for $0\leq i<k$, so $\delta_k(x)=(y+1)^k-y^k$ and 
$\delta_k(y)= -dy^k$. 

There is a more general class of automorphisms $\sigma$ of $R_x$, namely those for which $\sigma(x)=\lambda x$ and $\sigma(y)=y+p$ for some $\lambda\in \K^*$ and $p\in \K[x]$. For such an automorphism $\sigma$, it is again the case that if $p\not\equiv 1\bmod x\K[x]$ then every
$x$-locally inner $\sigma_p$-derivation is inner while if $p=dx+1$, where $d\in \K[x]$, then the space of $x$-locally inner $\sigma_p$-derivations is spanned by the inner $\sigma_p$-derivations and  the $x$-locally inner $\sigma_p$-derivations
$\delta_k$ induced by $x^{-1}y^k$, $k\geq 0$. Here $\delta_k(x)=(y+1)^k-\lambda y^k$ and 
$\delta_k(y)= -dy^k$.
\end{example}

\begin{example}\label{J} By \cite[Lemma 3.1]{BLOI}, $R_{\lambda x^2}\simeq R_{x^2}$ for all $\lambda\in \K^*$. The \textit{Jordan plane} 
$\mathcal{J}$ is presented by some authors, for example \cite{BLOI,NAG}, as $R_{x^2}$ and by others, for example \cite{ADP,BS}, as 
$R_{-\frac{1}{2}x^2}$. We shall use the latter presentation. Thus $yx-xy=-\frac{x^2}{2}$ and, in the notation of Example~\ref{Rh}, $f=-\frac{x}{2}$ and  
$\partial=-\frac{x^2}{2}\frac{d}{dx}$. If $p\notin x\K[x]$ then $p-f\notin x\K[x]$ and $\InnDer_{\sigma,x}(\mathcal{J})=\InnDer_\sigma(\mathcal{J})$ so we shall assume that 
$p\in x\K[x]$
and write
 $d=x^{-1}(p-f)$ as in Example~\ref{Rh}. From Example \ref{Rh}, we know that $\InnDer_{\sigma,x}(\mathcal{J})=\InnDer_\sigma(\mathcal{J})\oplus D$, where $D$ is spanned by the $x$-locally inner $\sigma$-derivations $\delta_k$ induced by $x^{-1}y^k$, $k\geq 0$, and that each $\delta_k(y)=-dy^k$. An easy induction shows that $fd/dx(\partial^m(x))=\frac{(m+1)!}{(-2)^{m+1}}x^{m+1}$ so
\[\delta_k(x)=\sum_{i=0}^{k-1}\frac{k!}{i!(-2)^{k-i}}x^{k-i}y^i.\] 

The Jordan plane is a significant example in the theory of $\N_0$-graded algebras.
The case where the Ore extension $S=\mathcal{J}[z;\sigma_p,\delta_k]$ is $\N_0$-graded, with $x$, $y$ and $z$ of degree $1$, is where $k=2$ and  $p=\beta x$ for some $\beta\in \K$. Here the defining relations are
\[yx = xy - \tfrac{1}{2}x^2,\quad zy = (y+ \beta x)z - (\tfrac{1}{2}+\beta)y^2\text{ and }
zx = xz - xy + \tfrac{x^2}{2}\]
and the elements $x$ and $w:=xz-y^2$ are normal.
\end{example} 

\begin{example} \label{double}
The \textit{Drinfeld double}, $\mathcal{D}$, \textit{of the Jordan plane} is a fascinating $\K$-algebra, with generators $x,y,g,\zeta,u,v$, that has been studied by Andruskiewitsch, Dumas and  Pena Pollastri in \cite{ADP} and by Brown and Stafford in \cite{BS}. The generators listed above may be adjoined, in various orders, either through an Ore extension or, in the case of $g$, by a skew Laurent extension. We shall assume that $\ch \K=0$.

The subalgebra generated by $x$ and $y$ is the Jordan plane $\mathcal{J}$ as in Example~\ref{J}.  
We shall present a construction of $\mathcal{D}$ from $\mathcal{J}$  in four steps, three of which involve $x$-locally inner skew derivations. This offers a different perspective to that in \cite{ADP} and  \cite{BS}, where
the existence of the appropriate skew derivations is established using the diamond lemma and a PBW basis. 

(i) In Example~\ref{J}, take $p=0$ and $k=1$ and let $\delta$ be the $x$-locally inner derivation $-2\delta_1$ induced by $-2x^{-1}y$. Then $\delta(x)=x$ and $\delta(y)=y$ and the subalgebra $\mathcal{L}$ generated by $x,y$ and $\zeta$ is $\mathcal{J}[\zeta;\delta]$, which has defining relations
\begin{equation}
\label{iL}
yx = xy - \tfrac{1}{2}x^2,\quad \zeta x=x\zeta+x\text{ and }\zeta y=y\zeta+y.
\end{equation}

(ii) To adjoin $g$, let $\sigma_x$ be the automorphism of $\mathcal{J}$ such that $\sigma_x(x)=x$ and $\sigma(y)=y+x$. Using \eqref{iL}, we see that $\sigma_x$ can be extended to an automorphism of $\mathcal{L}$ by setting $\sigma_x(\zeta)=\zeta$. The subalgebra $\mathcal{M}$ generated by $x$, $y$, $\zeta$, $g$ and $g^{-1}$ is then $\mathcal{L}[g^{\pm 1};\sigma_x]$, so that its defining relations are \eqref{iL} together with 
\begin{equation}
\label{iiM}
 gx=xg,\quad gy=(y+x)g,\quad g\zeta=\zeta g\text{ and }gg^{-1}=1=g^{-1}g.
\end{equation}

(iii) With a view to the adjunction of $u$, we may use \eqref{iL} and \eqref{iiM} to check that there is an automorphism $\sigma$ of $\mathcal{M}$ such that
\[\sigma(x)=x,\quad \sigma(y)=y,\quad\sigma(\zeta)=\zeta+1\text{ and }\sigma(g)=g.\] 
Observe  that $x$ is normal in $\mathcal{M}$ with normalizing automorphism $\phi$ such that
\[\phi(x)=x,\quad \phi(y)=y-\tfrac{1}{2}x,\quad\phi(\zeta)=\zeta+1\text{ and }\phi(g)=g,\]
whence, modulo $x\mathcal{M}$, $\ov{\phi^{-1}\sigma}=\id$.
Note also that $\mathcal{M}/x\mathcal{M}$ is a domain, namely the Laurent polynomial algebra in $\ov{g}$ over a copy of the enveloping algebra $R_x$ in Example~\ref{R1} with generators $\ov{y}$ and $\ov{\zeta}$. As $\ch(\K)=0$, $Z(R_x)=\K$ and  $Z(\mathcal{M}/x\mathcal{M})=\K[\ov{g}^{\pm 1}]$. 
By Corollary~\ref{deldercor}, $T(\mathcal{M},\sigma,x)=\pi^{-1}(Z(\ov{\mathcal{M}}))=x\mathcal{M}+\K[g^{\pm 1}]$. In particular there is an $x$-locally inner
$\sigma$-derivation induced by $x^{-1}t$, where $t=-2(g+1)$. By Proposition~\ref{delder2}(i), $\delta(x)=0=\delta(u)=\delta(\zeta)$ and  $\delta(y)=1-g$ so
 the relations for the Ore extension $\mathcal{N}:=\mathcal{M}[u;\sigma,\delta]$ are those from \eqref{iL} and \eqref{iiM}, together with 
\begin{equation}
\label{iiiN}
ux=xu,\quad uy-yu= 1-g,\quad u\zeta-\zeta u=u\text{ and }ug=gu.\end{equation} 

Proposition \ref{delder2}(iii) is applicable to check that the element $q:=ux+2(1+g)$ is normal in $\mathcal{N}$ with normalizing automorphism $\rho$   such that
\begin{equation}\label{rhoN}
\rho(x)=x,\;\rho(y)=y-\tfrac{1}{2}x,\;   \rho(\zeta)=\zeta,\;\rho(g)=g\text{ and }\rho(u)=u,
\end{equation} from which it follows that $g^{-1}q^2$ is central in $\mathcal{N}$.

(iv)
It follows from \eqref{iL}, \eqref{iiM} and \eqref{iiiN} that there is an automorphism $\tau$ of $\mathcal{N}$ such that
\[\tau(x)=x,\quad \tau(y)=y,\quad \tau(\zeta)=\zeta+1,\quad \tau(g)=g\text{ and }\tau(u)=u\] and that
$x$  remains normal in $\mathcal{N}$ with its
normalizing automorphism $\phi$  now extended so that
$\phi(u)=u$.
On $\mathcal{N}/x\mathcal{N}$, $\ov{\phi}=\ov{\tau}$, so every $x$-locally inner $\tau$-derivation of $\mathcal{N}$ is induced by $x^{-1}t$ for some $t\in \mathcal{N}$ such that 
$t+x\mathcal{N}\in Z(\mathcal{N}/x\mathcal{N})$.
Modulo $x\mathcal{N}$,
\[g\zeta\equiv\zeta g,\quad gy\equiv yg,\quad ug\equiv gu,\quad y\zeta\equiv(\zeta-1)y,\quad u\zeta\equiv(\zeta+1)u\text{ and }uy-yu\equiv 1-g,\]
from which it follows that  $\mathcal{N}/x\mathcal{N}$ is isomorphic to the ambiskew polynomial algebra $S$ of Example~\ref{xNambi}, with the elements
$c:=uy+(g-1)\zeta$ and $g$ being such that $c+x\mathcal{N}$ and $g+x\mathcal{N}$ correspond to the central elements $z$ and $g$  of $S$. 
Hence, as $\ch(\K)=0$, $T(\mathcal{N},\tau,x)=x\mathcal{N}+\K[g^{\pm 1},c]$. 
In particular, there is an $x$-locally inner $\tau$-derivation $\delta$ of $\mathcal{N}$ induced by 
$x^{-1}(2(g+1)-c+\tfrac{1}{2}xu\zeta)$. Proposition \ref{delder2}(i) can be applied to show that  
\begin{equation*}
\delta(x)=1-g+xu,\;\delta(y)=yu-g\zeta,\;\delta(g)=gu,\;\delta(\zeta)=0,\quad\delta(u)=-\tfrac{1}{2}u^2.
\end{equation*}
The algebra $\mathcal{D}$ is then the Ore extension $\mathcal{N}[v;\tau,\delta]$ and its defining relations are \eqref{iL}, \eqref{iiM}, \eqref{iiiN} together with
\begin{equation*}
vx=xv+1-g+xu,\;vy=yv+yu-g\zeta,\; vg-gv=gu,\;v\zeta=\zeta v+v,\text{ and }vu=uv-\tfrac{1}{2}u^2.
\end{equation*}
Proposition \ref{delder2}(iv) is applicable to show that the element \[s:=xv+uy+(g-1)\zeta-2(g+1)-\tfrac{1}{2}ux\zeta\] is normal in  $\mathcal{D}$ with  normalizing automorphism $\rho$ as in \eqref{rhoN} but now extended to $\mathcal{D}$ by setting $\rho(v)=v$. 
It follows that $g^{-1}q^2$, $g^{-1}s^2$ and $g^{-1}qs$ are central in $\mathcal{D}$. It is shown in \cite{BS} that $g^{-1}q^2$, $g^{-1}s^2$ and $g^{-1}qs$
generate the centre of $\mathcal{D}$.
\end{example}

\section{Skew derivations of the quantum torus}\label{sdonqt}
\subsection{Notation}
Throughout this section,  $n\geq 2$  is an integer, $Q=(q_{ij})$ is a multiplicatively antisymmetric  $n \times n$  matrix over $\K$, $\mathcal{T}$  is the quantum torus $\mathcal{O}_Q((\K^*)^n)$,  
 $\Lambda=(\lambda_1,\lambda_2,\dots,\lambda_n)\in (\K^*)^n$ and  $\sigma$ is the toric automorphism $\sigma_\Lambda$ of $\mathcal{T}$. 

The quantum torus $\mathcal{T}$ is $\Z^n$-graded with,  for $\db=(d_1,d_2,\dots,d_n)\in\Z^n$, the $\db$-component $\mathcal{T}_{\db}$ being the $1$-dimensional space $\K\x^{\db}$, where 
$\x^{\db}$ denotes the monomial
$x_1^{d_1}x_2^{d_2}\dots x_n^{d_n}$. 
For $1 \leq j \leq n$, let $\eb_j$ be the element of $\Z_n$ with $j$-component $1$ and all other components
$0$. Thus $\x^{\eb_j}=x_j.$

\subsection{Homogeneous skew derivations}
\begin{defn}\label{homogeneousder} Let $(G,+)$ be a commutative monoid and let $R=\oplus_{g\in G} R_g$ be a $G$-graded $\K$-algebra with a $G$-graded  automorphism $\sigma$.
For $g\in G$, a $\sigma$-derivation $\delta$  is \emph{homogeneous of weight} $g$
if $\delta(R_h)\subseteq R_{h+g}$ for all $h\in G$. 
We shall denote by $\Der_{\sigma,g}(A)$ the $\K$-vector space of all homogeneous $\sigma$-derivations of $R$ of weight $g$.
\end{defn}

\begin{lemma}\label{homcom}
With $G$, $R$ and $\sigma$ as in Definition~\ref{homogeneousder}, let $\delta$ be a $\sigma$-derivation of $R$, let $g\in G$ and let $\delta_g:R\rightarrow  R$ be the linear map such that, for all $h\in G$ and all $r\in R_h$, $\delta_g(r)=\pi_{g+h}(\delta(r))$, where $\pi_{g+h}$ is the projection from $R$ to $R_{g+h}$. Then $\delta_{g}$ is a homogeneous $\sigma$-derivation
of $R$ of weight $g$ and $\delta=\sum_{g\in G} \delta_g$. 
\end{lemma}
\begin{proof}
Let $h, j\in G$, let $r\in R_h$ and let $s\in R_j$. Then
\[\delta(rs)=\sigma(r)\delta(s)+\delta(r)s.\]
Noting that $rs\in R_{h+j}$, $\sigma(r)\in R_h$ and $s\in R_j$ and applying $\pi_{g+h+j}$, we see that
\[\delta_g(rs)=\sigma(r)\delta_g(s)+\delta_g(r)s.\]
By linearity, it follows that $\delta_{g}$ is a  $\sigma$-derivation
of $R$ and it is clearly homogeneous of weight $g$. It is also clear that $\delta=\sum_{g\in G} \delta_g$.
\end{proof} 

\begin{rmks}\label{sumhom}
 Let $\db\in \Z^n$ and let $\delta$ be a $\sigma$-derivation of the $\Z^n$-graded $\K$-algebra $\mathcal{T}$. 
Then \begin{enumerate}
\item $\delta$ is a homogeneous $\sigma$-derivation of weight $\db$ if and only if $\delta(x_j)\in \K \x^{\db+\eb_j}$ for $1\leq j\leq n$.
\item $\delta$ is  a unique sum of finitely many
homogeneous $\sigma$-derivations of $\mathcal{T}$, the weights that occur being those $\db\in \Z^n$ for which there exists $j$, with $1\leq j\leq n$, such that $\delta(x_j)$ has 
a non-zero component of degree $\db+\eb_j$. For each such $\db$, $\delta_{\db}$ will be called the
\emph{homogeneous component of} $\delta$ \emph{of weight} $\db$.
\end{enumerate}
\end{rmks}

\begin{defn}
We shall say that a subspace $V$ of $\Der_\sigma(\mathcal{T})$ is \emph{graded} if, for all $\delta\in \Der_\sigma(\mathcal{T})$, it is the
case that $\delta\in  V$ if and only if $\delta_{\db}\in V$ for all $\db \in \Z^n$. 

It is easy to see that $\InnDer_\sigma(\mathcal{T})$ is a graded subspace of $\Der_\sigma(\mathcal{T})$. 
\end{defn}

\begin{notn} 
For $\db = (d_1, d_2,\dots,d_n)\in\Z^n$ and $1\leq j\leq n$, let \[
q_j(\db) =
\prod_{k=1}^n q_{kj}^{d_k},\;r_j(\db) =
\prod_{k=j+1}^n q_{kj}^{d_k}\text{ and }s_j(\db) =
\prod_{\ell=1}^{j-1} q_{j\ell}^{d_\ell}.\] 
Note that $q_j$, $r_j$ and $s_j$ are group homorphisms from $\Z^n$ to $\K^*$ and that $q_j=r_js_j^{-1}$.
\end{notn}

\begin{lemma}\label{whendeltasigmainner}
Let $\db\in \Z^n$, let 
$\delta$ be  the inner $\sigma$-derivation of $\mathcal{T}$ induced by $\x^\db$ and let $\alpha$ be the inner automorphism of $\mathcal{T}$ induced by $(\x^\db)^{-1}$ so that, for all $t\in \mathcal{T}$, $\alpha(t)=\x^\db t(\x^\db)^{-1}$. Let $1\leq j\leq n$.
\begin{enumerate}
\item 
$
\x^{\db}x_j=r_j(\db)\x^{\db+\eb_j}$, $x_j\x^{\db}=s_j(\db)\x^{\db+\eb_j}$ and $\x^{\db}x_j=q_j(\x^{\db})x_j\x^{\db}$. 
\item   
$\delta(x_j)=(r_j(\db)-\lambda_j s_j(\db))\x^{\db+\eb_j}$ and so $\delta=0$ if and only if $q_j(\db)=\lambda_j$ for $1\leq j \leq n$.
\item $\alpha(x_j)=q_j(\x^{\db})x_j.$
\item
$\alpha=\sigma$ if and only if
$q_j(\db)=\lambda_j$ for $1\leq j \leq n$.
\end{enumerate}
\end{lemma}
\begin{proof} (i) is immediate from the defining relations for $\mathcal{T}$ and the definitions of $q_j(\db)$, $r_j(\db)$ and $s_j(\db)$ while (ii), (iii) and (iv) all
follow directly from (i).
\end{proof}

\begin{lemma}\label{homsderT}
Let $\delta$ be a homogeneous $\sigma$-derivation $\delta$ of $\mathcal{T}$ of weight $\db$ and, for $1 \leq j \leq n$,
let $a_j\in \K$ be such that $\delta(x_j)=a_j\x^{\db+\eb_j}$. Then
\begin{equation}\label{aa}
a_i(r_j(\db)-\lambda_js_j(\db))=a_j(r_i(\db)-\lambda_is_i(\db))\tag{$\dag$}
\end{equation}
when $1\leq i < j \leq n$.
\end{lemma}
\begin{proof} 
Let $1 \leq i<j \leq n$. Note that, as $i<j$, $s_i(\db+\eb_j)=s_i(\db)$, $r_j(\db+\eb_i)=r_j(\db)$, $r_i(\db+\eb_j)=q_{ij}^{-1}r_i(\db)$ and
$s_j(\db+\eb_i)=q_{ij}^{-1}s_j(\db)$. By Lemma~\ref{whendeltasigmainner}(ii),
\begin{eqnarray*}
\delta(x_ix_j)
&=& \lambda_ix_ia_j\x^{\db+\eb_j} + a_i\x^{\db+\eb_i}x_j
= (\lambda_ia_js_i(\db) + a_ir_j(\db))\x^{\db+\eb_i+\eb_j}\text{ and }\\
\delta(q_{ij}x_jx_i)
&=& q_{ij}(\lambda_jx_ja_i\x^{\db+\eb_i} + a_j\x^{\db+\eb_j}x_i)
= (\lambda_ja_is_j(\db) + a_jr_i(\db))\x^{\db+\eb_i+\eb_j}.
\end{eqnarray*} 
It follows that
\[\lambda_ia_js_i(\db) + a_ir_j(\db) = \lambda_ja_is_j(\db) + a_jr_i(\db)\]
and, on rearrangement, that \eqref{aa} holds.
\end{proof}

\begin{prop}\label{dichotomy} Let $\db\in \Z^n$.
\begin{enumerate}
\item
Suppose that $q_j(\db)\neq\lambda_j$
for some $j$ with $1 \leq j \leq n$. Then $\Der_{\sigma,\db}(\mathcal{T})$ is one-dimensional, spanned by the inner $\sigma$-derivation induced by $\x^{\db}$.

\item 
Suppose that  $q_j(\db)=\lambda_j$
for all $j$ with $1 \leq j \leq n$. Then $\sigma$ is the inner automorphism of $\mathcal{T}$ induced by  $\x^{-\db}$ and $\Der_{\sigma,\db}(\mathcal{T})$ is $n$-dimensional with a basis consisting of $n$ outer $\sigma$-derivations $\partial_j$, $1\leq j\leq n$, such that  $\partial_j(x_j) = \x^{\db+\eb_j}$ and $\partial_j(x_i) = 0$ if $i \neq j$.
\end{enumerate} 
\end{prop}
\begin{proof} (i) Let $\delta\in \Der_{\sigma,\db}(\mathcal{T})$. There exist $a_1,a_2,\dots,a_n\in \K$ such that
$\delta(x_i)=a_i\x^{\db+\eb_i}$. 
Let $b=a_j(r_j(\db)-\lambda_js_j(\db))^{-1}$. By Lemma~\ref{homsderT},
$a_i = b(r_i(\db)-\lambda_is_i(\db))$ for all $i$ and, by Lemma~\ref{whendeltasigmainner}(ii),
$\delta$ is the inner $\sigma$-derivation of $\mathcal{T}$ induced by $b\x^{\db}$.

(ii) By  Lemma~\ref{whendeltasigmainner}(iv), $\sigma$ is the inner automorphism of $\mathcal{T}$ induced by  $(\x^\db)^{-1}$.

It follows from \cite[0.8,p41]{FRR} that any Ore extension $A[z;\beta]$ of automorphism type 
has a derivation $\delta_z$ such that $\delta_z(A)=0$ and $\delta_z(z)=z$. Hence, for $1\leq j\leq n$, there is a derivation $\delta_j$ of the quantum space 
$\mathcal{O}_Q((\K^n)$ such that $\delta(x_j)=x_j$ and $\delta(x_i)=0$ if $i\neq j$. 
By \cite[Lemma 1.3]{gprimeskr}, $\delta_j$ extends to a derivation of $\mathcal{T}$. Let $\partial_j=r_j(\db)^{-1}x^{\db}\delta_j$. Applying Proposition~\ref{cog}(i) and Lemma~\ref{whendeltasigmainner}(i), we see that $\partial_j$ is a $\sigma$-derivation of $\mathcal{T}$ such 
that $\partial_j(x_j) = \x^{\db+\eb_j}$ and $\partial_j(x_i) = 0$ if $i \neq j$.By Lemma~\ref{whendeltasigmainner}(ii), each $\partial_j$ is outer. 

It is clear that the $\sigma$-derivations $\partial_j$ are linearly independent and that if $\delta$ is any homogeneous 
$\sigma$-derivation of weight $\db$ then $\delta=\sum_{i=1}^n a_i\partial_i$ where each $a_i\in\K$ is such that $\delta(x_i)=a_i\x^{\db+\eb_i}$.
\end{proof}

\begin{cor}\label{dichotomycor1}
If $\sigma$ is outer on $\mathcal{T}$ then every $\sigma$-derivation of $\mathcal{T}$ is inner.
\end{cor}
\begin{proof}
Suppose that $\sigma$ is outer. If $\delta$ is an  outer $\sigma$-derivation of $\mathcal{T}$ then it has an outer homogeneous  component $\delta_\db$ for some $\db\in \Z^n$. By Proposition~\ref{dichotomy}(i), $q_j(\db)=\lambda_j$ for all $j$ so, by Lemma~\ref{whendeltasigmainner}(iv), $\sigma$ is inner, contradicting the supposition. 
\end{proof}

\begin{cor}\label{dichotomycor2}
Suppose that $\sigma$ is inner, induced by $\x^{-\db}$. Then $\Der_\sigma(\mathcal{T})$ is the direct sum
of $\InnDer_\sigma(\mathcal{T})$ and the $Z(\mathcal{T})$-submodule $M$  of $\Der_\sigma(\mathcal{T})$ generated by
the $n$ $\sigma$-derivations $\partial_j$, $1 \leq j \leq n$, such that $\partial_j(x_j) = \x^{\db+\eb_j}$ and $\partial_j(x_i)=0$ if $i\neq j$.
\end{cor}
\begin{proof} By Lemma~\ref{whendeltasigmainner}(iv) and Proposition~\ref{dichotomy}(ii),
$\Der_{\sigma,\db}(\mathcal{T})$ is spanned by the $n$ $\sigma$-derivations $\partial_j$, $1 \leq j \leq n$. 
Let $\fb\in \Z^n$ and let $c=\x^{\fb}\x^{-\db}$. Then $\sigma$ is inner, induced by $\x^{-\fb}$, if and only if $c\in Z(\mathcal{T})$.
If  $c\in Z(\mathcal{T})$ then, for $1\leq j\leq n$, $c\partial_j(x_i)=0$ when $i\neq j$ and $c\partial_j(x_j)\in \K^*\x^{\fb+\eb_j}$ so, 
by Proposition~\ref{dichotomy}(ii), $\Der_{\sigma,\fb}(\mathcal{T})$ is spanned by the $n$ $\sigma$-derivations $c\partial_j$, $1 \leq j \leq n$. If $c\notin Z(\mathcal{T})$ 
then $\Der_{\sigma,\fb}(\mathcal{T})$ is $1$-dimensional, spanned by the inner $\sigma$-derivation induced by $\x^{\fb}$. Combining these, we see that $\Der_\sigma(\mathcal{T})=\InnDer_\sigma(\mathcal{T}) +M$. As $Z(\mathcal{T})$ is a graded subalgebra of $\mathcal{T}$, $M$ is a graded subspace of $\Der_\sigma(\mathcal{T})$. As $\InnDer_\sigma(\mathcal{T})$ is also graded,  so too is $\InnDer_\sigma(\mathcal{T})\cap M$. It then follows from Proposition~\ref{dichotomy}(ii) that $\Der_\sigma(\mathcal{T})=\InnDer_\sigma(\mathcal{T})\oplus M$.
\end{proof}

\begin{rmk}\label{dert}
Taking each $\lambda_j=1$ and each $d_i=0$ in Proposition~\ref{dichotomy}(ii), so that $\sigma=\id_\mathcal{T}$, $\sigma$-derivations are derivations and $\x^\db=1$, we see that $\Der(\mathcal{T})$ is the direct sum
of $\InnDer(\mathcal{T})$ and the $Z(\mathcal{T})$-submodule of $\Der(\mathcal{T})$ generated by
 $n$ derivations $\partial_j$, $1 \leq j \leq n$, such that $\partial_j(x_j)=x_j$ and $\partial_j(x_i)=0$ if $i\neq j$. 
\end{rmk}

\begin{cor}\label{Tdercor} Let $S =\mathcal{T}[x;\sigma; \delta]$ be an Ore extension of the quantum torus $\mathcal{T}$, where $\sigma$ is toric. Then $S$ is isomorphic either to
an Ore extension  $\mathcal{T}[y; \sigma]$ of automorphism type or to an Ore extension
$\mathcal{T}[y; \delta^\prime]$ of derivation type. In the latter case, there exist $z_1,z_2,\dots,z_n\in Z(\mathcal{T})$, such that 
$\delta^\prime=z_1\partial_1+z_2\partial_2+\dots+z_n\partial_n$,  where the derivations  $\partial_j$ are as in Remark~\ref{dert}.
\end{cor}
 
\begin{proof}
If $\sigma$ is outer then $\delta$ is inner by Corollary~\ref{dichotomycor1} and $S\simeq \mathcal{T}[y; \sigma]$ by Proposition \ref{cog}(iii). 
If $\sigma$ is inner then, by Proposition~\ref{cog}(ii), $S\simeq \mathcal{T}[y; \delta^\prime]$, where, by Remark~\ref{dert} the derivation $\delta^\prime$ is as stated.   
\end{proof}

\begin{cor}\label{quodivA1}
Let $S =\mathcal{T}[x;\sigma; \delta]$ be an Ore extension of the quantum torus $\mathcal{T}$, where $\sigma$ is toric. If $\sigma$ is inner and $\delta\neq 0$, then the quotient division algebra $D$ of $S$ contains a copy  
of the first Weyl algebra $A_1$. If $\sigma$ is outer or if $\delta=0$ then the quotient division algebra $D$ of $S$ is the quotient division algebra of a quantum torus $\mathcal{O}_{Q^\prime}((\K^*)^{n+1})$.
\end{cor}
\begin{proof}
If $\sigma$ is inner and $\delta\neq 0$ then, by Corollary \ref{Tdercor}, we may assume that $S=\mathcal{T}[x; \delta]$, that there exist $z_1,z_2,\dots,z_n\in Z(\mathcal{T})$ such that, for $1\leq i\leq n$, 
$\delta(x_i)=z_ix_i$ and that $z_i\neq 0$ for some $i$. 
Let $\mathcal{S}=Z(\mathcal{T})\backslash \{0\}$, the set of non-zero central elements of $\mathcal{T}$. Thus $\mathcal{S}$ is a left and right regular Ore set in $\mathcal{T}$ and
$\sigma(\mathcal{S})=\mathcal{S}$. Let $B$ be the localization $\mathcal{T}\mathcal{S}^{-1}=\mathcal{S}^{-1}\mathcal{T}$. By \cite[Lemmas 1.3 and 1.4]{gprimeskr}, 
$\delta$ extends to a derivation of $B$, $\mathcal{S}$ is a right and left regular Ore set in $S$ and $B[x;\delta]$ is the localization of $S$ at $\mathcal{S}$.
Let $u=z_ix_i\in U(B)$. By Proposition~\ref{cog}, $B[x;\delta]=B[z;\alpha,\delta^\prime]$, where $z=u^{-1}x$, $\alpha$ is the inner automorphism of $B$ induced by  $u^{-1}$, and $\delta^\prime$  is the $\sigma$-derivation $u^{-1}\delta$. Hence $\delta^\prime(x_i)=1$. Let $C=\K[x_i]$.  Then, on $C$,
$\alpha$ restricts to $\id_C$ and $\delta^\prime$ restricts to $d/dx_i$, whence the subalgebra generated by $z$ and $x_i$  is a copy of the Weyl algebra $A_1$.

If $\sigma$ is outer or if $\delta=0$ then,  as in the proof of Corollary~\ref{Tdercor}, $S\simeq \mathcal{T}[y; \sigma]$.  The skew Laurent polynomial algebra $\mathcal{T}[y^{\pm 1}; \sigma]$, is a quantum torus of the form $\mathcal{O}_{Q^\prime}((\K^*)^{n+1})$ and has the same quotient division algebra as $S$.
\end{proof}

\begin{rmk}\label{richardrmk} If $\ch \K=0$ then, by \cite[Th\'{e}or\`{e}me 5.2.2]{richard}, the quotient division algebra of a quantum torus $\mathcal{O}_Q((\K^*)^m)$ does not contain a copy of the Weyl algebra $A_1$.  Hence, in this case, the converse of 
Corollary~\ref{quodivA1} is true and the quotient division algebra $D$ of $S$ contains a copy of 
 $A_1$ if and only if $\sigma$ is inner and $\delta\neq 0$.
\end{rmk}

\section{Skew derivations of  quantum space}\label{sderqspace} 
Throughout this section, $n$, $Q$, $\mathcal{T}$, $\Lambda$ and $\sigma$ will be as in Section~\ref{sdonqt} and $\mathcal{A}$ will denote 
quantum space $\mathcal{O}_Q(\K^n)$.  For convenience,  the restriction of $\sigma$ to $\mathcal{A}$ will also be denoted $\sigma$. 
The $\Z^n$-grading on $\mathcal{T}$ restricts to an $\N_0^n$-grading on $\mathcal{A}$. Each $x_i$ is normal  in $\mathcal{A}$ with the normalizing automoprhism $\phi_{x_i}$ being such that $\phi_{x_i}(x_j)=q_{ji}x_j$ for $1\leq j\leq n$. 

By \cite[Lemma 1.3]{gprimeskr},
every $\sigma$-derivation $\delta$ of $\mathcal{A}$ extends uniquely to a $\sigma$-derivation of $\mathcal{T}$ which, for convenience, we shall also denote by $\delta$. Consequently 
\[\Der_\sigma(\mathcal{A})=\{\delta|_\mathcal{A}:\delta\in \Der_\sigma(\mathcal{T})\text{ and }\delta(\mathcal{A})\subseteq \mathcal{A}\}\] and, 
for $\db\in \Z^n$,
\[\Der_{\sigma,\db}(\mathcal{A})=\{\delta|_\mathcal{A}:\delta\in \Der_{\sigma,\db}(\mathcal{T})\text{ and }\delta(\mathcal{A})\subseteq \mathcal{A}\}.\] 
Note that 
$\Der_\sigma(\mathcal{A})$ is a graded subspace of $\Der_\sigma(\mathcal{T})$ so, to determine $\Der_\sigma(\mathcal{A})$, it will suffice to determine
$\Der_{\sigma,\db}(\mathcal{A})$ for each $\db\in \Z^n$.  
The following definition relates to those elements of 
$\Z^n\backslash \N_0^n$ that can occur as weights of non-zero homogeneous $\sigma$-derivations of $\mathcal{A}$. 
\begin{defn}
Let $\db\in \Z^n$ and let $1\leq j\leq n$. We say that $\db$ is $j$-\textit{exceptional} if $d_j=-1$ and $d_i\geq 0$ whenever $i\neq j$.
\end{defn}

\begin{lemma}\label{sigmaderR1} \begin{enumerate}
\item
$\Der_{\sigma,\db}(\mathcal{A})=\Der_{\sigma,\db}(\mathcal{T})$ for all $\db \in \N_0^n$. 
\item If $\db\in\Z^n\backslash \N_0^n$ and $\Der_{\sigma,\db}(\mathcal{A})\neq 0$ then $\db$ is $j$-exceptional for some $j$, $1\leq j\leq n$. 
\item If $\db$ is $j$-exceptional then the inner $\sigma$-derivation $\delta$  of $\mathcal{T}$ induced by $\x^{\db}$ restricts to an $x_j$-locally inner $\sigma$-derivation  of $\mathcal{A}$ if and only if $q_i(\db)=\lambda_i$ whenever $i\neq j$.
\end{enumerate}
\end{lemma}
\begin{proof}
(i) is clear. For (ii), let  $0\neq\delta\in \Der_{\sigma,\db}(\mathcal{A})$. As $\db \notin \N_0^n$, there exists $j$, $1\leq j\leq n$ such that 
$d_j<0$.   
For $i\neq j$, if
$\delta(x_i)\neq 0$ then $\delta(x_i)\in \K^*\x^{\db+\eb_i}$ so, as $d_j<0$,  $\delta(x_i)\notin \mathcal{A}$. Hence 
$\delta(x_i)=0$. As $\delta\neq 0$ we must have $\delta(x_j)\neq 0$. Then 
$\delta(x_j)\in \K^*\x^{\db+\eb_j}\cap \mathcal{A}$ so $d_i\geq 0$ whenever $i\neq j$ and  $d_j=-1$. Thus $\db$ is $j$-exceptional.

(iii) Note that $\db=x_j^{-1}t$ where $t=x_j\x^{\db}\in \mathcal{A}$. Let $\phi$ be the normalizing automorphism of $A$ induced by $x_j$ and let $\gamma=\phi^{-1}\sigma$. For $1\leq i\leq n$, $\phi(x_i)=q_{ij}x_i$, $\gamma(x_i)=q_{ji}\lambda_i x_i$ and $x_it\equiv q_i(\db+\eb_j)tx_i\bmod x_j\mathcal{A}$. 
As $q_i(\db+\eb_j)=q_{ji}q_i(\db)$ the result follows from Proposition~\ref{delder1}.
\end{proof}

\begin{prop}\label{sigmaderR2} 
Let  $1\leq j\leq n$, let $\db$ be $j$-exceptional and let $\delta$ be the inner derivation of $\mathcal{T}$ induced by $\x^\db$. 
\begin{enumerate}
\item 
If $q_j(\db)\neq\lambda_j$ and $q_i(\db)=\lambda_i$ whenever $i\neq j$,
then $\Der_{\sigma,\db}(\mathcal{A})$ is $1$-dimensional spanned by $\delta$ and $\delta(x_i)=0$ when $i\neq j$.
\item
If $q_i(\db)\neq\lambda_i$ for some $i\neq j$ then $\Der_{\sigma,\db}(\mathcal{A})=0$.
\item
If $q_i(\db)=\lambda_i$ for all $i$ then $\sigma$ is inner on $\mathcal{T}$ and $\Der_{\sigma,\db}(\mathcal{A})$ is $1$-dimensional spanned by an outer $\sigma$-derivation $\partial_j$ of $\mathcal{A}$ such that $\partial_j(x_j)=\x^{\db+\eb_j}$ and $\partial_j(x_i)=0$ when $i\neq j$.
\end{enumerate}
\end{prop}
\begin{proof}

(i),(ii) By Proposition~\ref{dichotomy}(i), $\Der_{\sigma,\db}(\mathcal{T})$ is one-dimensional, spanned by $\delta$. In (i), by Lemma~\ref{whendeltasigmainner}(ii), $\delta(x_j)\neq0$ and $\delta(x_i)=0$ when $i\neq j$, and, by Lemma~\ref{sigmaderR1}(iii), $\delta\in \Der_{\sigma,\db}(\mathcal{A})$ so $\Der_{\sigma,\db}(\mathcal{A})=\K \delta$. In (ii), $\delta\notin \Der_{\sigma,\db}(\mathcal{A})$ by Lemma~\ref{sigmaderR1} so 
$\Der_{\sigma,\db}(\mathcal{A})=0$.

(iii) In this case, $\delta=0$, by Lemma~\ref{whendeltasigmainner}(ii), and, by Proposition~\ref{dichotomy}(ii), $\Der_{\sigma,\db}(\mathcal{T})$ 
is $n$-dimensional with a basis consisting of the $n$ outer $\sigma$-derivations $\partial_k$, $1\leq k\leq n$, such that  $\partial_k(x_k) = \x^{\db+\eb_k}$ and $\partial_k(x_i) = 0$ if $i \neq k$. Clearly $\partial_j\in \Der_{\sigma,\db}(\mathcal{A})$. For $a_1,\dots,a_n\in \K$ if $a_i\neq 0$ then $(a_1\partial_1+\dots+a_n\partial_n)(x_i)\notin \mathcal{A}$ so
$\Der_{\sigma,\db}(\mathcal{T})=\K\partial_j$. 
\end{proof}

\begin{notn}
For $1\leq j\leq n$,
let $E_j$ denote the subspace of $\Der_{\sigma}(\mathcal{A})$ spanned by the homogeneous $\sigma$-derivations of $j$-exceptional weight and let $E=\oplus_{j=1}^n E_j$. This is consistent with the use of $E$ in the context of derivations in \cite{alevchamarie}. 
\end{notn}

\begin {rmks}
(i) By Proposition \ref{sigmaderR2}, $\Der_\sigma(\mathcal{A})=\InnDer_\sigma(\mathcal{A})\oplus F\oplus E$, where $F$ is spanned by the restrictions to $\mathcal{A}$ of 
those homogeneous $\sigma$-derivations that have weight $\db$ where $\db\in \N_0^n$ and $\sigma$ is inner on the quantum torus $\mathcal{T}$, induced by $\x^{-\db}$.

(ii) Any non-zero homogeneous $\sigma$-derivation $\delta$ of $j$-exceptional weight is either $x_j$-locally inner or $x_j$-locally conjugate to a derivation. In \cite{alevchamarie}, where $\sigma=\id_\mathcal{A}$, the former possibility does not arise. To see this, suppose that $\sigma=\id_\mathcal{A}$, so that $\lambda_i=1$ for all $i$, let $1\leq j\leq n$ and let $\delta$ be an $x_j$-locally inner derivation of $\mathcal{A}$ induced by $\x^\db$ where $d_j=-1$ and $d_i\geq 0$ if $i\neq j$. When $i\neq j$, $q_i(\db)=\lambda_i=1$, by Proposition~\ref{sigmaderR2}, so
$q_{ji}=\prod_{k\neq j}q_{ki}^{d_k}$, whence $q_j(\db)=\prod_{i\neq j}(\prod_{k\neq j}q_{ki}^{d_kd_i})=1$ as $q_{ik}=q_{ki}^{-1}$ when $i\neq k$ and $q_{ii}=1=q_{kk}$. By Lemma~\ref{whendeltasigmainner}(ii), $\delta=0$. 
\end{rmks}

\subsection{Normal elements} Suppose that $\sigma$ is outer on  $\mathcal{T}$ and let $\delta$ be an outer homogeneous $\sigma$-derivation  of $\mathcal{A}$. Thus there exist  $\db\in \Z^n$, $j$, $1\leq j \leq n$, and $a,b \in \K^*$, with $b=as_j(\db)$, such that $\db$ is $j$-exceptional, $q_j(\db)\neq \lambda_j$,  $q_i(\db)=\lambda_i$ when $i\neq j$, and  $\delta$ is $x_j$-locally inner on $\mathcal{A}$, induced by 
$x_j^{-1}t$, where $t=ax_j\x^{\db}=b\x^{\db+\eb_j}\in \mathcal{A}$.
Let  $S$ be the Ore extension $\mathcal{A}[z;\sigma,\delta]$, let $\mathcal{A}_j$ and $S_j$ be the localizations of $\mathcal{A}$ and $S$, respectively,  at the powers of $x_j$ and let 
$w=x_jz-t\in S$. Let $\phi$ be the normalizing automorphism induced  by $x_j$ and let $\gamma=\phi^{-1}\sigma$. Then $S_j$ is the Ore extension 
$\mathcal{A}_j[w;\gamma]$ of automorphism type and $w$ is normal in $S_j$, raising the question of whether $w$ is normal in $S$.
The next result uses Proposition~\ref{delder2} to give a positive answer to this question.

\begin{prop}\label{wnormal} Let the notation be as above,
\begin{enumerate}
\item $S_j = \mathcal{A}_j[w;\gamma]$,  where $\gamma$ is the automorphism of $\mathcal{A}_j$ such that, for  $1 \leq i \leq n$,
$\gamma(x_i) = \lambda_iq_{ji}x_i$.
\item $S_j$ is the selectively localized quantum space
$\mathcal{O}_{Q^\prime}(\K^{n+1})_{\langle j \rangle}$, where $x_{n+1}=w$ and $Q^\prime$ is the $(n+1)\times(n+1)$ multiplicatively antisymmetric matrix obtained from $Q$ by appending    
$q^\prime_{(n+1)i}=\lambda_iq_{ji}$  in the final row for $1\leq i\leq n$.
\item $\phi(t)=\sigma(t)$.
\item $w$ is normal in $S$ with $wz=\lambda_jzw$ and $wx_i = 
\lambda_iq_{ji}x_iw$,  $1\leq i\leq n$. 
\item If $1\leq i\leq n$ and $i\neq j$ then $x_i$ is normal in $S$.
\end{enumerate}
\end{prop}
\begin{proof} Note that, for $1 \leq i\leq n$, $\phi(x_i)=q_{ij}x_i$, so $\gamma(x_i)=\lambda_iq_{ji}(x_i)$. 

(i) is immediate from Proposition~\ref{delder2}(ii) and (ii) is immediate from (i).  

(iii)
To simplify the notation, we can assume, by renumbering the canonical generators and, accordingly, the
parameters $q_{ij}$ and $\lambda_i$,  that $j = 1$ and, replacing the generator $x_1$ by a scalar multiple, we may assume that $a=b=1$. 
Thus $t = \x^{\db+\eb_1}=x_2^{d_2}\dots x_n^{d_n}$.

For $2\leq i\leq n$, $\lambda_i=q_i(\db)=q_{i1}\prod_{k=2}^n q_{ki}^{d_k}$ so, as $q_{ii} = 1$ and $q_{ki}^{d_kd_i}q_{ik}^{d_id_k}=1$ when $k\neq i$,
\[\sigma(t)=\left(\prod_{i=2}^n \lambda_i^{d_i}\right)t=
\prod_{i=2}^n \left(q_{i1}^{d_i}\left(\prod_{k=2}^n q_{ki}^{d_k}\right)^{d_i}\right)t=
\left(\prod_{i=2}^n q_{i1}^{d_i}\prod_{i,k=2}^n q_{ki}^{d_kd_i}\right)t
=\left(\prod_{i=2}^n q_{i1}^{d_i}\right)t=\phi(t).\]

(iv) Given that $\sigma(x_j)=\lambda_jx_j$, 
this follows from (iii) and Proposition~\ref{delder2}(iii).

(v) If $i\neq j$ then $x_i\mathcal{A}=\mathcal{A}x_i$ and $\delta(x_i)=0$, so $zx_i=\lambda_i x_iz$. Hence $x_iS=Sx_i$. 
\end{proof}

\subsection{Linear combinations of locally inner $\sigma$-derivations}
Suppose that $\sigma$ is outer on $\mathcal{T}$ and let $\delta$ be a $\sigma$-derivation of $\mathcal{A}$ for which every homogeneous component is outer on $\mathcal{A}$. Each homogeneous component  of $\delta$
is $x_{j}$-locally inner, of $j$-exceptional weight, for possibly differing values of $j$. It is clear that $\delta$ is $x$-locally inner, induced by $x^{-1}v$ for some $v\in \mathcal{A}$,  where $x$ is  the product, in any order, of the different $x_j$'s that occur.
This subsection is aimed at showing that, as in the case of a single homogeneous $x_j$-locally inner $\sigma$-derivation, $\mathcal{A}[z;\sigma,\delta]$ has a distinguished normal element $xz-v$.
To do this we need to know how the condition that $\phi(t)=\sigma(t)$ behaves under the passage from $x_j$-locally inner $\sigma$-derivations to $x$-locally inner $\sigma$-derivations. The next lemma addresses this issue. We use the notation, from Subsection~\ref{regnorm}, whereby
$\phi_s$ denotes 
the normalizing automorphism of $\mathcal{A}$ induced by a regular normal element $s$ of $\mathcal{A}$.

\begin{lemma}\label{mixedlemma} 
Let $\db\in\Z^n$ be $j$-exceptional for some $j$, with $1\leq j\leq n$, and such that  $q_j(\db)\neq \lambda_j$ and, when $i\neq j$, $q_i(\db)=\lambda_i$. Let $\delta$ be the $x_j$-locally inner $\sigma$-derivation of $\mathcal{A}$ induced by $x_j^{-1}t$, where $t=\x^{\db+\eb_j}$. 
Let $m$ be an integer with $2\leq m\leq n$ and let $j_2,j_3,\dots,j_m$ be $m-1$ distinct integers such that,  for $2\leq k\leq m$, $1\leq j_k\leq n$   and $j_k\neq j$.
Let $y=x_{j_m}\dots x_{j_2}$, let $x=yx_j$  and let $t^\prime=yt$.  
Then $\delta$ is $x$-locally inner induced by $x^{-1}t^\prime$ and $\sigma(t^\prime)=\phi_x(t^\prime)$.
\end{lemma}
\begin{proof} As $x^{-1}t^\prime=x_j^{-1}t$, $\delta$ is $x$-locally inner induced by $x^{-1}t^\prime$. 

Let $\mu=q_j(\db+\eb_j)$, which, by Lemma~\ref{whendeltasigmainner}(i), is the eigenvalue of $t$ for $\phi_{x_j}$ and hence, by Proposition~\ref{wnormal}(iii), for $\sigma$.
Therefore $\sigma(t^\prime)=\mu \prod_{k=2}^{m}\lambda_{j_k}t^\prime$. It remains to compute $\phi_x(t^\prime)=\phi_x(y)\phi_x(t)$.

As $\phi_{x_j}(y)=\prod_{k=2}^{m} q_{j_kj}y$ and $\phi_y(y)=y$, we see that $\phi_x(y)=\prod_{k=2}^{m} q_{j_kj}y$.
For $2\leq k\leq m$, $j_k\neq j$ so, by Proposition~\ref{sigmaderR2}, $q_{j_k}(\db)=\lambda_{j_k}$. By Lemma~\ref{whendeltasigmainner}(ii),
\[\phi_{x_{j_k}}(t)=q_{j_k}(\db+\eb_j)t=q_{j_k}(\db)q_{j_k}(\eb_j)t=\lambda_{j_k}q_{jj_k}t.\] Hence 
$\phi_x(t)=\mu \prod_{k=2}^{m}\lambda_{j_k}q_{jj_k}t$. It now follows that $\phi_{x}(t^\prime)=\phi_x(y)\phi_x(t)=\mu\prod_{k=2}^m \lambda_{j_k}t^\prime=\sigma(t^\prime)$.
\end{proof}

\begin{prop}\label{mixed} 
Suppose that $\sigma$ is outer on $\mathcal{T}$ and let $\delta$ be a $\sigma$-derivation of $\mathcal{A}$.  There exist a normal element $w$ of the 
Ore extension $S=\mathcal{A}[z;\sigma,\delta]$ and a subset
$J=\{j_1,j_2,\dots,j_m\}\subseteq\{1,2,\dots,n\}$, such that, for $x=x_{j_1}x_{j_2}\dots x_{j_m}$, the localization $S_x$  of $S$ at $\{x^i:i\geq 0\}$ is a selectively localized quantized quantum space  $\mathcal{O}_{Q^\prime}(\K^{n+1})_{\langle j_1,j_2,\dots,j_m\rangle}$ with $x_{n+1}=w$.
Moreover, if $i\in \{1,2,\dots,n\}\backslash J$ then $x_i$ is normal in $S$.
\end{prop}
\begin{proof}
Let $\delta_1$ be the sum of those homogeneous components of $\delta$ that are inner on $\mathcal{A}$ and let $\delta_2$ be the sum of those homogeneous components of $\delta$ that are outer on $\mathcal{A}$. Then $\delta=\delta_1+\delta_2$, $\delta_1$ is inner on $\mathcal{A}$, induced by $a$
for some $a\in \mathcal{A}$, and, by Proposition~\ref{cog}(i) with $s=1$, $S=\mathcal{A}[z^\prime;\sigma,\delta_2]$, where $z^\prime=z-a$. Hence we may assume that $\delta_1=0$. If $\delta=0$ then the result holds with $m=0$, $x=1$ and $w=z$. Hence we may assume that $\delta\neq 0$.

Let $j_1,j_2,\dots,j_m$  be those values of $j$ such that $\delta$ has a non-zero homogeneous component that  is $x_{j}$-locally inner.
Let $\phi$ be the normalizing automorphism of $\mathcal{A}$ induced by $x$. Note that $\phi$ is independent of the order of the factors $x_{j_k}$, as different products are non-zero scalar multiples of each other, and that, for $1\leq i\leq n$, $\phi(x_i)=(\prod_{k=1}^{m}q_{j_ki})x_i$.

Let $\partial$ be a non-zero homogeneous component  of  $\delta$. By Lemma~\ref{sigmaderR1} and Proposition~\ref{sigmaderR2}, there exists $h$ with $1\leq h\leq m$, such that $\partial$
 is $x_{j_h}$-locally inner, induced by $\mu x_{j_h}^{-1}u$ for some $\mu\in \K^*$ and some monomial $u\in \mathcal{A}\backslash x_{j_h}\mathcal{A}$. 
Let $v=x\mu x_{j_h}^{-1}u$ so that $v\in \mathcal{A}$,  $\mu x_{j_h}^{-1}u=x^{-1}v$ and, by Lemma~\ref{mixedlemma}, 
$\phi(v)=\sigma(v)$.   Taking  $y=\sum v$ over the non-zero homogeneous components  of  $\delta$, we see that $\delta$ is  $x$-locally inner induced by 
$x^{-1}y$ and $\phi(y)=\sigma(y)$. By Proposition \ref{delder2}, $S_x=\mathcal{A}_x[w;\gamma]$,  where $w=xz-y$ and $\gamma$ is the automorphism of $\mathcal{A}_x$ such that, for $1 \leq i  \leq n$,  
$\gamma(x_i)=\lambda_i(\prod_{k=1}^{m}q_{j_ki}) x_i$, and $w$ is normal in $S$. With $x_{n+1}=w$, $S_x$ is  the selectively localized  quantum space  $\mathcal{O}_{Q^\prime}(\K^{n+1})_{\langle j_1,j_2,\dots,j_m\rangle}$, 
where $Q^\prime$ is the
multiplicatively antisymmetric $(n+1)\times (n+1)$ matrix obtained from $Q$ by appending $q_{(n+1)i}=\lambda_i(\prod_{k=1}^{m}q_{j_ki})$ in the final row for $1\leq i\leq n$.

If $i\notin J$ then $\delta_k(x_i)=0$ for $1\leq k\leq \ell$, so $\delta(x_i)=0$ and $zx_i=\lambda_i x_iz$. Also, for $1\leq j\leq n$, $x_ix_j=q_{ij}x_jx_i$. Hence $x_i$ is normal in $S$.
\end{proof}

\section{Skew derivations of selectively localized quantum spaces}\label{sdersellocqspace} 
The methods applied in Section~\ref{sderqspace} to the quantum space $\mathcal{A}$ are easily adapted to  selectively localized quantum spaces 
$\mathcal{A}_{\langle i_1,\dots,i_k\rangle}$, where $1\leq k<n$. The integer $k$ will be fixed throughout the section and, to simplify the notation, we  assume  that $i_\ell=\ell$ for $1\leq \ell\leq k$ and denote by  $\mathcal{L}$ the localization $\mathcal{A}_{\langle 1,\dots,k\rangle}$ of the quantum space $\mathcal{A}$ at the powers of $x_1x_2\dots x_k$. Then $\mathcal{L}$ is $(\Z^k\times \N_0^{n-k})$-graded and, as we have seen for $\mathcal{A}$,
\[\Der_\sigma(\mathcal{L})=\{\delta\in \Der_\sigma(\mathcal{T}):\delta(\mathcal{L})\subseteq \mathcal{L}\}.\]
We present below, without proof, the analogues of the results of Section~\ref{sderqspace}. The proofs are essentially the same as in Section~\ref{sderqspace}, taking account of the distinction between those canonical generators that have been selected for inversion and those that have not.
We will need the following generalization of the term $j$-\textit{exceptional} when working with $\mathcal{L}$.

\begin{defn}
Let $\db\in \Z^n$ and let $k+1\leq j\leq n$. We say that $\db$ is $j$-\textit{exceptional} if $d_j=-1$ and $d_i\geq 0$ whenever $k+1\leq i\leq n$ and $i\neq j$.
\end{defn} 
The next four results are analogues of Lemma~\ref{sigmaderR1} and Propositions \ref{sigmaderR2}, \ref{wnormal} and \ref{mixed}.
\begin{lemma}\label{sigmaderRk1}
\begin{enumerate}\item
If $\db \in \Z^k \times \N_0^{n-k}$ then  $\Der_{\sigma,\db}(\mathcal{L})=\Der_{\sigma,\db}(\mathcal{T})$.
\item  
If $\db\in\Z^n\backslash (\Z^k \times \N_0^{n-k})$ and $\Der_{\sigma,\db}(\mathcal{L})\neq 0$ then $\db$ is $j$-exceptional for some $j$, 
$k+1\leq j\leq n$.
\item If $\db$ is $j$-exceptional then the inner $\sigma$-derivation $\delta$  of $\mathcal{T}$ induced by $\x^{\db}$ restricts to an $x_j$-locally inner $\sigma$-derivation  of $\mathcal{L}$ if and only if $q_i(\db)=\lambda_i$ whenever $i>k$ and $i\neq j$.
\end{enumerate}
\end{lemma}

\begin{prop}\label{sigmaderRk2} 
Let  $k+1\leq j\leq n$ and let $\db$ be $j$-exceptional. 
\begin{enumerate}
\item
If $q_i(\db)=\lambda_i$ for all $i$ with $k+1\leq i\leq n$ then $\Der_{\sigma,\db}(\mathcal{L})$ is $1$-dimensional spanned by the $\sigma$-derivation $\partial_j$ of $\mathcal{L}$ such that $\partial_j(x_j)=\x^{\db+\eb_j}$ and $\partial_j(x_i)=0$ when $i\neq j$.
\item 
If $q_j(\db)\neq\lambda_j$ and $q_i(\db)=\lambda_i$ whenever $i\neq j$ and $k+1\leq i\leq n$,
then $\Der_{\sigma,\db}(\mathcal{L})$ is $1$-dimensional spanned by the $x_j$-locally inner $\sigma$-derivation $\delta$ of $\mathcal{L}$ induced by $\x^{\db}$.
\item
If $q_i(\db)\neq\lambda_i$ for some $i\neq j$ with $k+1\leq i\leq n$ then $\Der_{\sigma,\db}(\mathcal{L})=0$.
\end{enumerate}\end{prop}

\begin{prop}\label{wnormalsel} Let $\db\in\Z^k\times \N_0^{n-k}$ be $j$-exceptional for some $j$, $k+1\leq j\leq n$, such that 
$q_j(\db)\neq\lambda_j$ and, when $1\leq i\leq n$ and $i\neq j$,  $q_i(\db)=\lambda_i$. Let $\delta$ be the $x_j$-locally inner $\sigma$-derivation of $\mathcal{L}$ induced by $a\x^{\db}$ for some $a\in \K^*$. 
Let $S$ be the Ore extension $\mathcal{L}[z; \sigma,\delta]$ and let $\mathcal{L}_j$ and $S_j$ denote the localizations of $\mathcal{L}$ and
$S$ respectively at $\{x_j^i:i\in \N\}$.
Let $t=ax_j\x^{\db}=a\rho \x^{\db+\eb_j}\in \mathcal{L}$, where $\rho=s_j(\db)$,  
and let $w = x_jz-t$. Then
\begin{enumerate}
\item $S_{j} = \mathcal{L}_{j}[w;\gamma]$  where $\gamma$ is the automorphism of $\mathcal{L}_{j}$ such that, for $1 \leq i \leq n$, 
$\gamma(x_i) = \lambda_iq_{ji}x_i$.
\item  With $x_{n+1}=w$, $S_j$ is the selectively localized quantum space
$\mathcal{O}_{Q^\prime}(\K^{n+1})_{\langle 1,2,\dots,k, j \rangle}$ where $Q^\prime$ is the $(n+1)\times(n+1)$ multiplicatively antisymmetric matrix obtained from $Q$ by appending    
$q^\prime_{(n+1)i}=\lambda_iq_{ji}$ for  $1\leq i\leq n$ in the final row.
\item If $\phi$ is the normalizing automorphism of $\mathcal{L}$ induced by $x_j$ then $\phi(t) = \sigma(t)$.
\item $w$ is normal in $S$ with $wz = \lambda_jzw$ and $wx_i = 
\lambda_iq_{ji}x_iw$,  $1\leq i\leq n$.
\item For $k+1\leq i\leq n$, if $i\neq j$ then $x_i$ is a normal non-unit in $S$.
\end{enumerate}
\end{prop}

\begin{prop}\label{mixedsel} 
Suppose that $\sigma$ is outer on $\mathcal{T}$ and let $\delta$ be a $\sigma$-derivation of $\mathcal{L}$.  There exist a normal element $w$ of the Ore extension $S=\mathcal{L}[z;\sigma,\delta]$
and a subset
$I=\{j_1,j_2,\dots,j_m\}\subseteq\{k+1,k+2,\dots,n\}$, such that, for $x=x_{j_1}x_{j_2}\dots x_{j_m}$, the localization $\mathcal{L}_x$  of $\mathcal{L}$ at $\{x^i:i\geq 0\}$ is a selectively localized quantized quantum space  $\mathcal{O}_{Q^\prime}(\K^{n+1})_{\langle 1,2,\dots,k,j_1,j_2,\dots,j_m\rangle}$ with $x_{n+1}=w$.
Moreover, if $i\in \{k+1,k+2,\dots,n\}\backslash I$ then $x_i$ is a normal non-unit in $S$.
\end{prop}

\section{Examples}
\subsection{The case $n=1$} The skew derivations in the case $n=1$ are well-known, there being classifications of Ore extensions of $\K[y]$ in \cite{alevdumasinv,AVV,BLOI}. However it may be helpful to interpret them in terms of Proposition \ref{sigmaderR2}.
Write $y$ for $x_1$ so that $\mathcal{A} = \K[y]$,  $\mathcal{T}=\K[y^{\pm 1}]$, $Q=(1)$, $\Lambda=(\lambda)$ for some 
$\lambda\in \K^*$ and $\sigma(y)=\lambda y$. 
Let $d\in \Z$ and note that $q_1(y^d)=1$. 
  
If $\lambda=1$ then $\sigma=\id_{\mathcal{A}}$ and, in accordance with Proposition~\ref{dichotomy}(ii), 
$\Der_{\sigma,d}(\mathcal{T})=\Der_{d}(\mathcal{T})$ is $1$-dimensional spanned by 
 $y^{d+1}d/dy$. In accordance with Proposition \ref{sigmaderR2}(i),  $\Der_{\sigma,d}(\mathcal{A})$ is $1$-dimensional spanned by 
 $y^{d+1}d/dy$ if $d\geq -1$. If $d<-1$ then $\Der_{\sigma,d}(\mathcal{A})=0$. 

On the other hand, if $\lambda\neq 1$ then, in accordance with Propositions~\ref{dichotomy}(i) and \ref{sigmaderR2}, 
$\Der_{\sigma,d}(\mathcal{T})$ is $1$-dimensional spanned by the inner $\sigma$-derivation $\delta_d$ of $\mathcal{T}$ induced by $y^d$, $\Der_{\sigma,d}(\mathcal{A})=0$
 if $d<-1$ and $\Der_{\sigma,d}(\mathcal{A})=\Der_{\sigma,d}(\mathcal{T})$ if $d\geq -1$. Here $\delta_d$ is inner on $\mathcal{A}$ if $d\geq 0$ and $y$-locally inner if $d=-1$. In the latter case, $\delta_{-1}(y)=1-\lambda$ and the Ore extension $\mathcal{A}[x;\sigma,\delta_1]$ is the quantum disc as in Example~\ref{qdisc}, where $\lambda$ was written $q$.

\subsection{The case $n=2$}\label{qplanesigmader}

In this subsection, we describe the homogeneous $\sigma$-derivations of the quantum torus $\mathcal{T}=\mathcal{O}_q((\K^*)^2)$ and the quantum plane 
$\mathcal{A}=\mathcal{O}_q(\K^2)$. We shall write $x$ and $y$, for $x_1$ and $x_2$ respectively. Thus $\sigma(x)=\lambda_1 x$ and $\sigma(y)=\lambda_2 y$.

\begin{examples}\label{ntwoT} 
Consider the quantum torus $T=\mathcal{O}_q((\K^*)^2)$ and let $\db=(i,j)\in \Z^2$. Then  $q_1(\x^\db)=q^{-j}$ and $q_2(\x^\db)=q^{i}$.
If $\lambda_1\neq q^{-j}$ or $\lambda_2\neq q^{i}$ then, by Proposition~\ref{dichotomy}(i), 
$\Der_{\sigma,\db}(\mathcal{T})$ is $1$-dimensional, spanned by the inner $\sigma$-derivation induced by $\x^\db$ which is such that
$x\mapsto (q^{-j}-\lambda_1)x^{i+1}y^j$   and $y\mapsto (q^{i}-\lambda_2)x^{i}y^{j+1}$.  

If $\lambda_1=q^{-j}$ and $\lambda_2=q^{i}$ then $\sigma$ is inner, induced by $\x^{\db}$ and, by Proposition~\ref{dichotomy}(ii), $\Der_{\sigma,\db}(\mathcal{T})$  is $2$-dimensional, spanned by $\sigma$-derivations $\partial_x$ and $\partial_y$ such that 
$\partial_x(x)=x^{i+1}y^{j}$, $\partial_x(y)=0=\partial_y(x)$ and  $\partial_y(y)=x^{i}y^{j+1}$. 
In this case, if $q$ is not a root of unity then $i$, $j$ and $\db$ are unique
and $\Der_\sigma(\mathcal{T})=\InnDer_\sigma(\mathcal{T})\oplus \Der_{\sigma,\db}(\mathcal{T})$ while if $q$ is a primitive $m$th root of unity, for some $m\geq 1$, then $i$, $j$ and $\db$ are unique up to congruence modulo $m$,
$Z(\mathcal{T})=\K[(x^m)^{\pm 1},(y^m)^{\pm 1}]$ and $\Der_\sigma(\mathcal{T})=\InnDer_\sigma(\mathcal{T})\oplus Z(\mathcal{T})\Der_{\sigma,\db}(\mathcal{T})$ for any choice of $\db$. 
\end{examples}

\begin{examples}\label{ntwoR} Here we consider the homogeneous $\sigma$-derivations $\delta$ of the quantum plane $\mathcal{A}=\mathcal{O}_q(\K^2)$ and comment on some of the  resulting Ore extensions $S=\mathcal{A}[z;\sigma,\delta]$. 
Let $\db=(i,j)\in \Z^2$.  By Lemma \ref{sigmaderR1} and Proposition~\ref{sigmaderR2}, the space $\Der_{\sigma,\db}(\mathcal{A})=0$ except in the cases below, where we begin with the inner or locally inner $\sigma$-derivations of $\mathcal{A}$.

\noindent(i) If $i,j\geq 0$ and either $\lambda_1\neq q^{-j}$ or $\lambda_2\neq q^{i}$ then $\Der_{\sigma,\db}(\mathcal{A})$ is $1$-dimensional spanned by the inner $\sigma$-derivation $\delta$ induced by $x^iy^j$ and,  by Proposition~\ref{cog}, $S$ is  of automorphism type. Indeed it is quantum $3$-space $\mathcal{O}_Q(\K^3)$ with canonical generators $x$, $y$ and $z-x^{i}y^{j}$ and parameters $q_{12}=q, q_{31}=\lambda_1$ and $q_{32}=\lambda_2$. 
 
\noindent(ii) If $i=-1$, $j\geq 0$,  $\lambda_1\neq q^{-j}$ and $\lambda_2=q^{-1}$ then $\Der_{\sigma,\db}(\mathcal{A})$ is $1$-dimensional spanned by
the $x$-locally inner $\sigma$-derivation $\delta$, induced by $(q^{-j}-\lambda_1)^{-1}x^{-1}y^j$, for which $\delta(x)=y^j$ and $\delta(y)=0$.
Here the defining relations of $S$ are
\[xy=qyx,\quad zx=\lambda_1 xz+y^j,\quad zy=q^{-1}yz.\] 
The element $y$ is normal in $S$, but $x$ is not, and the element 
$w=xz-(q^{-j}-\lambda_1)^{-1}y^j$ is normal and 
such that $yw=wy$, $zw=\lambda_1 wz$ and $wx=\lambda_1 xw$. If $j>0$,
$S/yS\simeq \mathcal{O}_{q^{-j}}(\K^2)$   while, if $j=0$ then, as $\lambda_1\neq 1$, $S/yS$ is, in the notation of Example~\ref{qdisc}, the quantized Weyl algebra $A_1^{\lambda_1}$.

Note that $yw$ is central in $S$ if $\lambda_1=q$. More generally if $\lambda_1^b=q^c$ for some $b,c\in \Z$ then $y^bw^c$ is central in the quantum torus $\mathcal{T}^\prime$ generated by $x^{\pm1}$, $y^{\pm1}$ and $w^{\pm1}$ which is the localization of $S$ at $\{x^iy^jw^k:i,j,k\in \N_0\}$. 

On the other hand, suppose that $\lambda_1^b\neq q^c$ for all $b,c\in \Z$. By Proposition \ref{McCPcrit}, $\mathcal{T}^\prime$ is simple from which it follows, as $y$ and $w$ are normal in $S$, that every non-zero prime ideal of $S$ contains $y$, $w$ or $x^i$ for some $i\in \N_0$. A straightforward induction shows that \[\partial(x^i)=q^{-(i+1)}\qnum{i}{p}x^{i-1}y^j,\] where $p=\lambda_1q^j$ and, for $i\in \N$, 
$\qnum{i}{p}=(p^i-1)/(p-1)$. Note that $p-1\neq 0$ and $p^i-1\neq 0$ by the conditions on $\lambda_1$ and $q$. It follows from this  that the height one 
prime ideals of $S$ are $yS$ and $wS$.

\noindent(iii) If $j=-1$, $i\geq 0$, $\lambda_1=q$ and $\lambda_2\neq q^{i}$ then $\Der_{\sigma,\db}(\mathcal{A})$ is $1$-dimensional spanned by 
the $y$-locally inner $\sigma$-derivation $\delta$, induced by $(q^i-1)^{-1}y^{-1}x^{i}$, for which $\delta(x)=0$  and $\delta(y)=x^i$.
The Ore extension $S$ is similar to that in the previous case, there being symmetry involving $x$ swapping with $y$ and $q$ with $q^{-1}$.

\noindent(iv)
If $i,j\geq 0$, $\lambda_1=q^{-j}$ and $\lambda_2=q^{i}$ then $\Der_{\sigma,\db}(\mathcal{A})$ is $2$-dimensional spanned by the $\sigma$-derivations $\delta_x$ and $\delta_y$ such that $\delta_x(x)=x^{i+1}y^j$, $\delta_x(y)=0$, $\delta_y(x)=x^{i}y^{j+1}$, and  $\delta_y(x)=0$. 
Here the automorphism $\sigma$ is inner on the quantum torus $\mathcal{T}$. Let $\delta=\partial_x+\partial_y$. Then $S$  has defining 
relations 
\begin{equation*}xy=qyx,\quad zx=q^{-j}xz+x^{i+1}y^{j},\quad zy=q^iyz+x^{i}y^{j+1}\end{equation*} 
and the elements $x$ and $y$ are both normal in $S$ which has an $\N_0$-grading in which  $\deg(x)=1=\deg(y)$ and $\deg(z)=i+j$.
The factor $S/xS$ is isomorphic to $\mathcal{O}_{q^i}(\K^2)$ if $i>0$ and to the Ore extension  $\K[y][z;y^{j+1}d/dy]$ if $i=0$. Similarly, 
$S/yS\simeq \mathcal{O}_{q^{-j}}(\K^2)$  if $j>0$ and $S/yS\simeq \K[x][z;\tau x^{i+1}d/dx]$ if $j=0$.

Suppose that $q$ is not a root of unity, in which case $i$ and $j$ are uniquely determined by $\lambda_1$ and $\lambda_2$, and that $\ch(\K)=0$. 
By Proposition~\ref{McCPcrit}, the quantum torus $\mathcal{T}$ is simple. As $\sigma$ is inner on $\mathcal{T}$, induced by $x^{i}y^{j}$,  we see, by applying Proposition~\ref{cog}, that $\mathcal{T}[z;\sigma,\delta]\simeq \mathcal{T}[z^\prime;\delta^\prime]$ 
for the outer derivation  $\delta^\prime$ of $\mathcal{T}$ such that $\delta^\prime(x)=x$ and $\delta^\prime(y)=y$. By \cite[Proposition 2.1]{GW} or \cite[Proposition 1.8.4]{McCR},
the Ore extension $\mathcal{T}[z;\sigma,\delta]$ is simple. It follows, using \cite[Proposition 2.1.16(v)]{McCR}, that the height one primes of $S$  are $xS$ and $yS$. Also, from the above description of $S/xS$, we can deduce that if $i>0$ then
the height two primes of $S$ containing $xS$ are $xS+yS$ and $xS+zS$  and that if $i=0$ the only such prime of $S$ is $xS+yS$. The situation for the height two primes of $S$ containing $yS$ is similar.

\noindent(v) If $i=-1$, $j\geq 0$, $\lambda_1=q^{-j}$ and $\lambda_2=q^{-1}$ then $\Der_{\sigma,\db}(\mathcal{A})$ is $1$-dimensional spanned by the $\sigma$-derivation 
$\delta$ such that $\delta(x)=y^j$ and $\delta(y)=0$. Here the defining relations for $S=\mathcal{A}[z;\sigma,\delta]$  are as in (ii), but with $\lambda_1=q^{-j}$, and 
$\delta$ is locally conjugate to a derivation and is not $x$-locally inner.
Again $y$ is normal, but $x$ is not. If q is not a root of unity and
$\ch\K = 0$ then the localization of $S$ at the powers of $y$ is simple, see  [22, Example 6.12(ii)].
It follows, using [26, Proposition 2.1.16(v)], that $yS$ is the unique height one prime ideal of
$S$. 

\noindent(vi) By symmetry, if $j=-1$, $i\geq 0$, $\lambda_1=q$ and $\lambda_2=q^{i}$ then $\Der_{\sigma,\db}(R)$ is $1$-dimensional spanned by 
 the  $\sigma$-derivation $\delta$ such that $\delta(x)=0$ and $\delta(y)=x^i$.
\end{examples}

\begin{examples}\label{overlap}
Perhaps the most interesting $\sigma$-derivations of $\mathcal{A}$ are those that occur when both $\lambda_2=q^{-1}$ and $\lambda_1=q$. Taking linear combinations of the homogeneous $\sigma$-derivations in Examples~\ref{ntwoR}(ii,iii,v,vi),  we see that, given $g(y)\in \K[y]$ and $f(x)\in \K[x]$,  there is a $\sigma$-derivation $\delta_{f,g}$  of $\mathcal{A}$
 such that $\delta_{f,g}(x)=g(y)$ and $\delta_{f,g}(y)=f(x)$. The defining relations for the Ore extension $S=\mathcal{A}[z;\sigma,\delta_{f,g}]$ are then
\[xy=qyx,\quad zx=qz+g(y),\quad zy=q^{-1}yz+f(x).\] When $q$ is not a root of unity, these are  the $xy$-locally inner $\sigma$-derivations that were discussed in Examples~\ref{qplane} and  first observed in \cite[Theorem 6.2(2)(b)]{AlBr}.
\end{examples}

\subsection{Skew derivations of the selectively localized quantum plane}
\begin{examples}\label{ntwoRx}
Here, with the notation as in Examples~\ref{ntwoR}, we consider the selectively localized quantum plane $\mathcal{L}=\mathcal{A}_{\langle 1\rangle}$, that is the localization of $\mathcal{A}=\mathcal{O}_q(\K^2)$
at the powers of $x$. 
Applying Propositions \ref{dichotomy} and \ref{sigmaderRk2}, the non-zero spaces $\Der_{\sigma,(i,j)}(\mathcal{L})$ are as follows:
\begin{enumerate}
	\item if $j\geq 0$ and $\lambda_1\neq q^{-j}$ or $\lambda_2\neq q^{i}$ then $\Der_{\sigma,(i,j)}(\mathcal{L})$ is $1$-dimensional spanned by the inner $\sigma$-derivation induced by $x^iy^j$;
\item  if $j\geq 0$ and $\lambda_1=q^{-j}$ and $\lambda_2=q^{i}$, then $\Der_{\sigma,(i,j)}(\mathcal{L})$ is $2$-dimensional spanned by the $\sigma$-derivations $\partial_x$ and $\partial_y$ such that $\partial_x(x)=x^{i+1}y^{j}$, $\partial_x(y)=0$, $\partial_y(x)=0$ and $\partial_y(y)=x^{i}y^{j+1}$;
\item if $j=-1$, $\lambda_1=q$ and $\lambda_2\neq q^{i}$, then $\Der_{\sigma,(i,j)}(\mathcal{L})$ is $1$-dimensional spanned by the $y$-locally inner $\sigma$-derivation $\delta$ induced by $y^{-1}x^{i}$, which is such that $\delta(x)=0$ and $\delta(y)=(q^i-\lambda_2)x^i$;
\item  if $j=-1$, $\lambda_1=q$ and $\lambda_2=q^{i}$, then $\Der_{\sigma,(i,j)}(\mathcal{L})$ is $1$-dimensional spanned by the $\sigma$-derivation $\partial_y$ such that $\partial_y(x)=0$ and $\partial_y(y)=x^{i}$, which is $y$-locally conjugate to a derivation and not $y$-locally inner.
\end{enumerate}
\end{examples}

\begin{example}\label{Uq}
Probably the best known Ore extension of a selectively localized quantum plane is the \emph{quantized enveloping algebra} $U_q(\mathfrak{sl}_2)$, see \cite[Chapters 
I.3 and I.4]{BGl}.  For $q\in\K\backslash\{0,1,-1\}$, $U_q(\mathfrak{sl}_2)$ is the $\K$-algebra with generators 
$K^{\pm 1}, E$ and $F$ and relations 
\[
KK^{-1}=1=K^{-1}K, \quad KE=q^2EK,\quad FK=q^2KF,\quad FE-EF=\frac{K^{-1}-K}{q-q^{-1}}.
\]
Writing $x_1=K$, $x_2=E$ and $x_3=F$, this is the Ore extension 
$\mathcal{L}[x_3;\sigma,\delta]$, where $\mathcal{L}$ is the selectively localized quantum plane $\mathcal{O}_{q^2}(\K^2)_{\langle 1\rangle}$, $\sigma$ is the toric automorphism such that $\sigma(x_1)=q^2x_1$ and $\sigma(x_2)=x_2$ while $\delta$ is the $\sigma$-derivation   such that 
$\delta(x_1)=0$ and $\delta(x_2)=(q-q^{-1})^{-1}(x_1^{-1}-x_1)$. 
Here $\delta$ has two $x_2$-locally inner homogeneous components, induced by 
 $x_2^{-1}x_1^{-1}$ and 
 $x_2^{-1}x_1$, and is $x_2$-locally inner induced by
$-(q-q^{-1})^{-2}x_2^{-1}(qx_1^{-1}+q^{-1}x_1)$. In accordance with Proposition~\ref{delder2}(iii), with $s=x_2$ and  $\phi_s=\sigma$, the element 
\[w=x_2x_3+(q-q^{-1})^{-2}(qx_1^{-1}+q^{-1}x_1),\] which is denoted $C_q$ in \cite[I.4.5]{BGl}, is central in $U_q(\mathfrak{sl}_2)$.
The localization of $U_q(\mathfrak{sl}_2)$ at the powers of $x_2$ is the 
selectively localized quantum space $\mathcal{O}_{Q}(\K^3)_{\langle 1,2\rangle}$ where the canonical generators are $x_1^{\pm1},x_2^{\pm1}$ and $w$ and 
\[
Q =
\left(\begin{array}{ccccc}
1 &q^2 &1\\
q^{-2}& 1 &1\\
1&1&1\\
\end{array}\right).\]

\end{example}

\subsection{Skew derivations of the quantum disc}
Let $S$ be the quantum disc $\K[y][x;\sigma,\delta]$, where $\sigma(y)=qy$ and $\delta(y)=1
-q$, let $\lambda\in \K^*$ and let $\tau\in \aut(S)$ be such that $\tau(x)=
\lambda x$ and 
$\tau(y)=\lambda^{-1}y$. 
Recall, from Example~\ref{qdisc}, that the element $w:=yx-1=q^{-1}(xy-1)$ is normal in $S$ with normalizing automorphism $\phi$ such that $\phi(x)=qx$ and $\phi(y)=q^{-1}y$.  Note that 
$y=(w+1)x^{-1}=x^{-1}(qw+1)$ and the localization $\mathcal{L}$ of $S$ at the powers of $x$ is the selectively localized quantum plane $\mathcal{O}_q(\K^2)_{\langle 1\rangle}$, with canonical generators $x$, $x^{-1}$ and $w$ and the relation $xw=qwx$. The extension of $\tau$ to $\mathcal{\mathcal{L}}$ is toric with $\tau(x)=\lambda x$ and $\tau(w)=1$. By \cite[Lemma 1.3]{gprimeskr}, every $\tau$-derivation of $S$ extends uniquely to $\mathcal{L}$ so
$\Der_\tau(S)=\{\delta\in \Der_\tau(\mathcal{L}):\delta(S)\subseteq S\}$. 
As  $\Der_\tau(\mathcal{L})$ is known through the results of Section 6, this suggests an approach to the determination of $\Der_\tau(S)$. 
In this subsection we take this approach to determine $\Der_\tau(S)$ in the case where $q$ is not a root of unity.
This approach may be applicable to other Ore extensions $S=\mathcal{A}[z;\sigma,\delta]$, where $\mathcal{A}$ is a quantum affine space $\mathcal{O}_Q(\K^n))$, $\sigma$ is a toric automorphism of $\mathcal{A}$,  $\delta$ is $x_j$-locally inner for some $j$, so that the localization $\mathcal{L}$ of  $S$ at the powers of $x_j$ is a 
 selectively localized quantum space $\mathcal{L}$, and  $\tau$ is a toric automorphism of $\mathcal{L}$ that restricts to an automorphism of $S$.

We begin by applying Examples~\ref{ntwoRx} with $\lambda_1=\lambda$ and $\lambda_2=1$,  and with $w$ in place of $y$, to determine $\Der_\tau(\mathcal{L})$. There are three cases, labelled (i),(ii),(iii) in the following lemma.
\begin{lemma}\label{Fpm}
\begin{enumerate}
\item Suppose that $\lambda=q$. 
\begin{enumerate}
\item
For $m\in \Z\backslash\{0\}$, there is a $w$-locally inner $\tau$-derivation $\delta_m$ of $\mathcal{L}$, induced by $w^{-1}x^m$, such that 
$\delta_m(x)=0$, $\delta_m(w)=(q^m-1)x^{m}$ and $\delta_m(y)=(q^m-1)x^{m-1}$. 
\item
There is a 
$\tau$-derivation $\partial$ on $\mathcal{L}$ such that $\partial(x)=0$ and $\partial(w)=1$.
\item 
$\Der_\tau(\mathcal{L})=\InnDer_\tau(\mathcal{L})\oplus F^+\oplus F^-\oplus \K\partial$, 
where $F^+$ has basis $\{\delta_m: m>0\}$ and  $F^-$ has basis $\{\delta_m: m<0\}$.
\end{enumerate}
\item Suppose that $\lambda=q^{-j}$ for some $j\geq 0$. There are $\tau$-derivations  $\partial_x$ and $\partial_w$ of $\mathcal{L}$  such that $\partial_x(x)=xw^j$, $\partial_x(w)=0$, $\partial_w(x)=0$ and $\partial_w(w)=w^{j+1}$ and
$\Der_{\tau}(\mathcal{L})=\InnDer_{\tau}(\mathcal{L})\oplus \K\partial_x\oplus \K\partial_w$
\item Suppose that $\lambda\neq q^k$ for all $k\leq 1$. Then $\Der_{\tau}(\mathcal{L})=\InnDer_{\tau}(\mathcal{L})$. 
\end{enumerate}
\end{lemma}
\begin{proof} Apart from the details for $\delta_m(y)$ in  (i)(a), this is immediate from Examples~\ref{ntwoRx}. In (i)(a), for $m\in \Z\backslash\{0\}$,
$(q^m-1)x^m=\delta_m(w)=\delta_m(yx-1)=\delta_m(y)x$, whence, $\mathcal{L}$ being a domain, $(q^m-1)x^{m-1}=\delta_m(y)$.
\end{proof}

\begin{lemma}\label{H}
For $m\in \Z$ and $n\in \N_0$, let $\zeta_{m,n}$ denote the inner $\tau$-derivation of $\mathcal{L}$ induced by $x^my^n$ and let $H$ be the subspace  of $\Der_\tau(\mathcal{L})$ spanned by
$\{\zeta_{m,n}: m<0\}$. 
\begin{enumerate}
\item
	$\InnDer_\tau(\mathcal{L})=\InnDer_\tau(S)\oplus H$,	
	and $\InnDer_\tau(S)$ is spanned by the inner $\tau$-derivations $\zeta_{m,n}$  where $m\geq 0$ and $n\geq 0$.
	\item For $m\in \Z$ and $n\geq 0$,
	$\zeta_{m,n}(w)=
(q^{n-1}-q^{m-1})(x^{m+1}y^{n+1}-x^{m}y^{n})$.
\end{enumerate}
\end{lemma}
\begin{proof}
(i) This is an immediate consequence of the fact that $\mathcal{L}$ and $S$ have bases $\{x^my^n:m\in \Z, n\in \N_0\}$ and $\{x^my^n:m,n\in \N_0\}$ respectively

(ii) As $\tau(w)=w$, $wx=q^{-1}xw$, $yw=q^{-1}wy$ and $w=q^{-1}(xy-1)$, we see that \begin{equation*}\zeta_{m,n}(w)=x^my^nw-wx^my^n=(q^{n}-q^{m})x^{m}wy^n=
(q^{n-1}-q^{m-1})(x^{m+1}y^{n+1}-x^{m}y^{n}).
\end{equation*}

\end{proof}
The next step is to determine the space of $w$-locally inner $\sigma$-derivations of $S$.

\begin{prop}\label{discwloc} Suppose that $q$ is not a root of unity.
\begin{enumerate}
\item If $\lambda\neq q$ then every $w$-locally inner derivation is inner. 
\item If $\lambda=q$ then, in the notation of \ref{T}, $T(S,\tau,w)$ is spanned by $\{x^m:m\geq 0\}\cup \{y^m:m\geq 1\}$. For each $m\in \N$, 
the $w$-locally inner  
derivations $\delta_m$, induced by $w^{-1}x^m$ and $\gamma_m$, induced by $w^{-1}y^m$, of $S$ are such that $\delta_m(x)=0$, $\delta_m(y)=(q^m-1)x^{m-1}$, $\gamma_m(y)=0$ and $\gamma_m(x)=(q^{1-m}-q)y^{m-1}$. In particular $\delta_0=\gamma_0=0$. 
\end{enumerate}
\end{prop}
\begin{proof}
(i) As $S/wS\simeq \K[x^{\pm 1}]$ is commutative and $\phi^{-1}\sigma(x)=\lambda q^{-1}x$, this is immediate from Proposition~\ref{delder1}, given that $q$ is not a root of unity.

(ii) In this case, $S/wS\simeq \K[x^{\pm 1}]$ is commutative, with $y+wS=(x+wS)^{-1}$,  and $\phi^{-1}\sigma(x)=x$ so, by Proposition~\ref{deldercor},  
$T(S,\tau,w)$ is spanned by $\{x^m:m\geq 0\}\cup \{y^m:m\geq 1\}$.
By Lemma~\ref{Fpm}(i)(a), $\delta_m(x)=0$ and $\delta_m(y)=(q^m-1)x^{m-1}$. Replacing  $x$ by $y$ and $q$ by $q^{-1}$ 
we see that $\gamma_m(y)=0$ and 
 $\gamma_m(w)=(q^{-m}-1)y^{m}$. As \[(q^{-m}-1)y^{m}=\gamma_m(w)=\gamma_m(yx-1)=q^{-1}y\gamma_m(x),\] and $\mathcal{L}$ is a domain, it follows that
$\gamma_m(x)=(q^{1-m}-q)y^{m-1}$.
\end {proof}

\begin{prop}\label{skewderqdisc} 
Suppose that $q$ is not a root of unity. 
\begin{enumerate}
\item If $\lambda=q$ then  $\Der_\tau(S)=\InnDer_\tau(S)\oplus F$, where $F=F^+\oplus G^+$, $F^{+}$ is as in Lemma~\ref{Fpm}(i)(c) and $G^+$ is the subspace of $\Der_\tau(S)$ with basis   $\{\gamma_m: m>0\}$.
\item If $\lambda=q^{-j}$ for some $j\geq 0$ then $\Der_\tau(S)=\InnDer_\tau(S)\oplus \K\partial_x$, where $\partial_x$ is a $\tau$-derivation of $S$ such that $\partial_x(x)=xw^j$, $\partial_x(w)=0$ and $\partial_x(y)=-q^jw^jy$.
\item If $\mu\neq q^k$ for all $k\leq 1$ then $\Der_\tau(S)=\InnDer_\tau(S)$.
\end{enumerate}
\end{prop}
\begin{proof} 
(i) Let $D=\InnDer_\tau(S)\oplus F$. Then $D\subseteq \Der_\tau(S)$ and we need to show that $D=\Der_\tau(S)$.

Let $I$ be the ideal $w\mathcal{L}=\mathcal{L}w$. Modulo $I$, $x$ and $y$ are inverses of each other so, for $m\geq 1$
$x^{-m}-y^m\in I$ and hence  $w^{-1}(x^{-m}-y^m)\in \mathcal{L}$. It follows that $\delta_{-m}\equiv \gamma_m\bmod \InnDer_\tau(\mathcal{L})$ and hence that
$\Der_\tau(\mathcal{L})=\InnDer_\tau(\mathcal{L})\oplus F\oplus \K\partial=D\oplus H\oplus \K\partial$.  As $D\subseteq \Der_\tau(S)$,
 it suffices to show that $(H\oplus \K\partial) \cap \Der_\tau(S)=0$.

Let $\delta\in H$ and $a\in \K$ and let $\gamma= \delta+a\partial\in (H\oplus \K\partial) \cap \Der_\tau(S)$. Then $\delta$ is inner on $\mathcal{L}$, induced by  
$v=\sum_{m<0,n\geq 0} a_{m,n}x^my^n\in V$, where $a_{m,n}\in\K$ for $m<0$ and $n\geq 0$, and 
$a_{m,n}=0$ for all but finitely many pairs $(m,n)$. Suppose that $v\neq 0$. There exists $m<0$ such that $a_{m,n}\neq 0$ for some $n\geq 0$ and $a_{p,k}=0$ whenever $p<m$ and $k\geq 0$. 
Choose any $n\geq 0$ such that $a_{m,n}\neq 0$. By Lemma~\ref{H}(ii) and as $\partial(w)=1$, 
the coefficient of $x^{m}y^{n}$ in $\gamma(w)$, relative to the basis $\{x^my^n:m\in \Z, n\in \N_0\}$, is $a_{m,n}(q^{m-1}-q^{n-1})$. As $n\geq 0$, $m<0$ and $q$ is not a root of unity, $a_{m,n}(q^{m-1}-q^{n-1})\neq 0$, whence $\gamma(w)\notin S$ and $\gamma\notin \Der_\tau(S)$. Thus $v=0$, $\delta=0$, $\gamma=a\partial$ and 
$\gamma(y)=a\partial(wx^{-1}+x^{-1})=ax^{-1}$,
whence $a=0$ and $\gamma=0$. Thus $(H\oplus \K\partial) \cap \Der_\tau(S)=0$ and $\Der_\tau(S)=D$. 

(ii) Let $\partial_x$ and $\partial_w$ be as in Lemma~\ref{Fpm}(ii), so that 
$\Der_{\tau}(\mathcal{L})=\InnDer_{\tau}(\mathcal{L})\oplus \K\partial_x\oplus \K\partial_w=\InnDer(S)\oplus H\oplus \K\partial_x\oplus \K\partial_w$.
As  \[0=\partial_x(qw)=\partial_x(xy-1)=q^{-j}x\partial_x(y)+q^jxw^jy,\] we see that $\partial_x(y)=-q^jw^jy\in S$, whence, as $\partial_x(x)\in S$ also, $\partial_x\in \Der_\tau(S)$. 
It follows that $\Der_{\tau}(\mathcal{L})=D\oplus\K\partial_w\oplus H$, where $D=\InnDer_{\tau}(S)\oplus \K\partial_x\subseteq \Der_{\tau}(S)$ and $H$ is defined as in the proof of (i). 

Note that $\partial_w(y)=\partial_w(x^{-1}(qw+1))=q^{j+1}x^{-1}w^{j+1}\in S$. Suppose that $\partial_w(y)\in S$. Then $w^{j+1}\in xS$. Hence, as $w^{j+1}=ww^j=q^{-1}(xy-1)w^j$, $w^j\in xS$. Repeating the argument, $w^{j-1}\in xS,\dots,w\in xS, 1\in xS$  which is false. Thus $\partial_w(y)\notin S$.  It follows, essentially as in (i) but with $\partial_w$ replacing $\partial$, that 
$(H\oplus \K\partial_w) \cap \Der_\tau(S)=0$ and $\Der_\tau(S)=D$.

(iii) In this  case we see, from Examples~\ref{ntwoRx}, that $\Der_\tau(\mathcal{L})=\InnDer_\tau(\mathcal{L})=\InnDer_{\tau}(S)\oplus H$. By a similar argument as in (i,ii) but without an analogue  of $\partial$ or $\partial_w$, $H\cap \Der_\tau(S)=0$ so  $\Der_\tau(S)=\InnDer_\tau(S)$.
 \end{proof}

\begin{example} Suppose that $\lambda=q$ and let $\delta_1$ and $\gamma_1$ be as in Proposition~\ref{discwloc}(ii).
Let $\delta=q^{-1}\delta_1$. Then  
$\delta(y)=1-q^{-1}$ and $\delta(x)=0$, and, writing $x_1$ for $x$ and $x_2$ for $y$, 
the defining relations for $S[x_3;\tau,\delta]$ are
\begin{equation*}x_1x_2-qx_2x_1=1-q,\quad x_2x_3-qx_3x_2 = 1 - q,\quad x_3x_1 = qx_1x_3,\end{equation*}
the same  
as we saw for the linear connected quantized Weyl algebra $L_q^3$,  as an Ore  extension of $\mathcal{O}_q(\K^2)$, 
 in Examples~\ref{overlap}.
Unlike the situation in Examples~\ref{overlap},  the $\tau$-derivation $\delta$ is $q^{-1}$-skew, by Remark~\ref{qskewness}(ii), as
$\delta\tau(x)=0=q^{-1}\tau\delta(x)$ and 
$\delta\tau(y)=q^{-1}(1-q^{-1})=q^{-1}\tau\delta(y)$. 

Another connected quantized algebra, the \emph{cyclic connected quantized Weyl algebra} $C_q^3$ \cite{fishjordan}, arises as an Ore extension of the quantum disc.
Taking $\delta=q^{-1}\delta_1+\gamma$, so that $\delta(y)=1-q^{-1}$ and $\delta(x)=1-q$, 
 the defining relations for $S[x_2;\tau,\delta]$ are
\[x_1x_2-qx_2x_1 = 1-q,\quad x_2x_3-qx_3x_2=1-q,\quad
x_3x_1-qx_1x_3 = 1-q,\]
which are the defining relations for $C_q^3$.
In contrast to the situation for $L_3^q$, the $\tau$-derivation $\delta$ here is neither $q$-skew nor $q^{-1}$-skew as  
$\delta\sigma(y)=q^{-1}\sigma\delta(y)$ whereas $\delta\sigma(x)=q\sigma\delta(x)$.
\end{example}

\begin{example} When $j=0$ in Proposition~\ref{skewderqdisc}(ii), $\tau=\id_S$ so $\Der(S)=\InnDer(S)\oplus \K\partial_x$ where 
$\partial_x$ is a derivation of $S$ such that $\partial_x(x)=x$, $\partial_x(w)=0$ and $\partial_x(y)=-y$. 
The Ore extension $S[z;\partial_x]$ has 
defining relations
\[xy-qyx=1-q,\quad zx-xz=x,\quad zy-yz=-y.\]
Let $S_w$ denote the localization of $S$ at the powers of the normal element $w$ and suppose that $\ch \K=0$.
The algebra $S_w$ is known to be simple, for example see \cite[8.4 and 8.5]{gprimeskr}. Suppose that $\partial$ is inner on $S_w$, induced by $w^{-j}s$, where $s\in S$ and $j\geq 0$. 
Thus \begin{equation}\label{partialinner}
w^{-j}sx-xw^{-j}s=x\text{ and } w^{-j}sw=ww^{-j}s.
\end{equation} There is a $\Z$-grading on $S_w$ for which $(S_w)_0=\K[w^{\pm 1}]$ and, for $n>0$, $(S_w)_n=x^n\K[w^{\pm 1}]$ and
$(S_w)_{-n}=y^n\K[w^{\pm 1}]$. Taking terms of degree $1$ in \eqref{partialinner}, we may assume that $s\in \K[w^{\pm 1}]$. 
Thus $\partial_x$ is inner  on the selectively localized quantum plane $B$ generated by $x$ and $w^{\pm1}$. 
Taking terms of degree $(1,0)$ in \eqref{partialinner}, in the $\N_0\times \Z$-grading on $B$, we may assume that $s\in B_{0,j}=w^j\K$. 
But then $w^{-j}s\in \K$ and $w^{-j}sx-xw^{-j}s=0$, giving a contradiction. Therefore $\partial$ is not inner on $S_w$, and it follows from \cite[Theorem 1.8.4]{McCR} that the Ore extension $S_w[z;\partial_x]$ is simple. It then follows, using [26, Proposition 2.1.16(v)], that $wS[z;\partial_x]$ is the unique height one prime ideal of 
$S[z;\partial_x]$. By \cite[Lemma 1.4]{gprimeskr}, $S[z;\partial_x]/wS[z;\partial_x]$ is isomorphic to the Ore extension $\K[t^{\pm 1}][v;td/dt]$ which, using  \cite[Theorem 1.8.4]{McCR}, is easily seen to be simple, whence $wS[z;\partial_x]$ is the unique non-zero prime ideal of $S[z;\partial_x]$.
\end{example}

\subsection{The case $n\geq 3$}
\begin{rmk}\label{chooselambda}
Let $Q=(q_{ij})$ be a multiplicatively antisymmetric $n\times n$ matrix  over $\K$, and let $\db\in \Z^n$ be such that $\db$ is $j$-exceptional for some $j$, $1\leq j\leq n$. Then there exists $\Lambda=(\lambda_1,\dots,\lambda_n)\in \K^*$ such that $\Der_{\sigma_\Lambda,\db}(\mathcal{A})$ is $1$-dimensional spanned by a $\sigma_\Lambda$-derivation $\delta$ which, depending on the choice of $\lambda_j$, may be $x_j$-locally inner or $x_j$-locally conjugate to a derivation. 
To see this, first choose $\lambda_i=q_i(\db)$ for $i\neq j$, and let $\lambda_j\in \K^*$. 
If $\lambda_j\neq q_j(\db)$ then $\Der_{\sigma_\Lambda,\db}(\mathcal{A})=\K\delta$ for the $x_j$-locally inner $\sigma_\Lambda$-derivation $\delta$ induced by 
$(r_j(\db)-\lambda_js_j(\db))^{-1}\x^\db$,
but if 
$\lambda_j=q_j(\db)$ then $\Der_{\sigma_\Lambda,\db}(\mathcal{A})=\K\delta$ for a $\sigma_\Lambda$-derivation $\delta$ that is  $x_j$-locally conjugate to a derivation. 
In both cases,
$\delta(x_j)=\x^{\db+\eb_j}$ and $\delta(x_i)=0$ if $i\neq j$.

For example, let $\mathcal{A}$ be the single-parameter quantum affine space $\mathcal{O}_q(\K^n)$, where $q_{ij}=q$ when $i<j$, and let $\db$ be such that $d_1=-1$ and $d_i=1$  if $i>1$. Then 
$q_1(\db)=q^{1-n}$ while $q_i(\db)=q^{2i-3-n}$ if $i>1$. Suppose that $\lambda_i=q^{2i-3-n}$ for $i\geq 2$ and let $\lambda_1\in \K^*$.    If $\lambda_1\neq q^{1-n}$ then $\Der_{\sigma_\Lambda,\db}(\mathcal{A})=\K\delta$, where $\delta$ is $x_1$-locally inner, induced by 
$(r_j(\db)-\lambda_1s_j(\db))^{-1}x_1^{-1}x_2x_3\dots x_n$, whereas if $\lambda_1=q^{1-n}$ then $\Der_{\sigma_\Lambda,\db}(\mathcal{A})=\K\delta$ where $\delta$ is $x_1$-locally conjugate to a derivation. In both cases,
$\delta(x_1)=x_2x_3\dots x_n$ and $\delta(x_i)=0$ if $i>1$.
\end{rmk}

\begin{example}\label{2by2qmat}
Here we illustrate Remark \ref{chooselambda} in the context of  multiparameter quantum matrices. Let $n=3$, let  $\db=(-1,1,1)$, let $q,p,r,\lambda_1,\lambda_2,\lambda_3\in \K^*$ and
let
\[
Q_3=
\left(\begin{array}{ccc}
1 &q &p\\
q^{-1}& 1 &r\\
p^{-1}& r^{-1} &1
\end{array}\right)\text{ and }\Lambda=(\lambda_1,\lambda_2,\lambda_3).\]
Let $\sigma$ be the toric automorphism $\sigma_\Lambda$. 
Here $q_1(\db)=(pq)^{-1}$, $q_2(\db)=(qr)^{-1}$ and $q_3(\db)=rp^{-1}$.
By Proposition~\ref{sigmaderR2}, $\Der_{\sigma,\db}(\mathcal{A})\neq 0$ if and only if $\lambda_2=(rq)^{-1}$ and $\lambda_3=rp^{-1}$, in which case 
$\Der_{\sigma,\db}(\mathcal{A})$ is $1$-dimensional spanned by a $\sigma$-derivation that is $x_1$-locally inner if and only if  $\lambda_1\neq (pq)^{-1}$, giving rise to the two cases discussed below.

Case (a). Suppose that $\lambda_2=(rq)^{-1}$, $\lambda_3=rp^{-1}$ and $\lambda_1\neq (pq)^{-1}$. Then 
$\Der_{\sigma,\db}(\mathcal{A})$ is $1$-dimensional spanned by the 
inner $\sigma$-derivation $\delta$ induced by 
$qx_1^{-1}x_2x_3$ which is such that $\delta(x_2)=\delta(x_3)=0$ and $\delta(x_1)=(p^{-1}-\lambda_1q)x_2x_3$. The defining relations for 
$S:=\mathcal{A}[x_4;\sigma,\delta]$ are 
\begin{eqnarray}\label{pqr}
x_1x_2&=&qx_2x_1,\quad x_1x_3=px_3x_1,\quad x_2x_3=rx_3x_2,\\
\label{pqrplus}
x_4x_1&=&\lambda_1x_1x_4+(p^{-1}-\lambda_1q)x_2x_3,\quad x_4x_2=(rq)^{-1}x_2x_4,\quad x_4x_3=rp^{-1}x_3x_4.
\end{eqnarray}
In accordance with Proposition~\ref{wnormal}, the element $d=x_1x_4-qx_2x_3$ is normal in $S$ with
\[dx_1=\lambda_1x_1d,\quad dx_2=r^{-1}x_2d,\quad dx_3=rpq^{-1}x_1d\text{ and }dx_4=\lambda_1x_4d,\] and the localization of $S$ at the powers of $x_1$ is the selectively localized quantum space 
$\mathcal{O}_{Q_4}(\K^4)_{\langle 1\rangle}$  for the matrix
\[
Q_4=
\left(\begin{array}{cccc}
1 &q &p&\lambda_1^{-1}\\
q^{-1}& 1 &r&r^{-1}\\
p^{-1}& r^{-1} &1&(rp)^{-1}q\\
\lambda_1& r& rpq^{-1}&1
\end{array}\right).\]

Such algebras $S$ occur as subalgebras of the \textit{multiparameter algebra} $\mathcal{O}_{\lambda,P} (M_n(\K))$ \textit{of} $n\times n$ \textit{quantum matrices} \cite[p16]{BGl}, where $P$ is an $n\times n$ multiplicatively antisymmetric matrix  
over $\K$ and $\lambda\in \K^*$. The algebra $\mathcal{O}_{\lambda,P} (M_n(\K))$ is generated by $\{x_{ij}:1\leq i,j\leq n\}$ subject to the relations
obtained from \eqref{pqr} and \eqref{pqrplus} on setting $x_1=x_{ij}, x_2=x_{im}, x_3=x_{\ell j}, x_4=x_{\ell m}, q=p_{mj}, p=\lambda^{-1}p_{i\ell}, 
r=pq^{-1}=\lambda^{-1}p_{i\ell}p_{jm}$ and $\lambda_1=\lambda^{-1}p^{-1}q^{-1}=p_{i\ell}p_{jm}$, for $1\leq i
<\ell\leq n$ and $1\leq j<m\leq n$.

When $\lambda_1=1$, $p=q$ and $r=1$, so that $q^2\neq1$, $S$ is the \textit{algebra} $\mathcal{O}_q(M_2(\K))$  \textit{of} $2\times 2$ \textit{quantum matrices} and $d$ is the \textit{quantum determinant}, which is well-known to be central, see \cite[p6]{BGl}.

Case (b). Suppose that $\lambda_2=(rq)^{-1}$, $\lambda_3=rp^{-1}$ and $\lambda_1=(pq)^{-1}$. Then 
$\Der_{\sigma,\db}(\mathcal{A})$ is $1$-dimensional spanned by the $\sigma$-derivation $\partial$ of $\mathcal{A}$ such that $\partial(x_1)=x_2x_3$ and  
$\partial(x_2)=\partial(x_3)=0$.
The defining relations for $S=\mathcal{A}[x_4;\sigma,\partial]$ are \eqref{pqr} and
\begin{equation*}
x_4x_1=(pq)^{-1}x_1x_4+x_2x_3,\quad x_4x_2=(rq)^{-1}x_4x_2,\quad x_4x_3=rq^{-1}x_3x_4.
\end{equation*}

Although the relations look superficially like those in Case (a), this algebra is quite different, at least when $\ch \K=0$. 
As $(x_3^{-1}x_2^{-1}x_4)x_1-x_1(x_3^{-1}x_2^{-1}x_4)=1$,
the quotient division algebra of $S$ contains a copy of the first Weyl algebra $A_1$ whereas in Case (a) the quotient division algebra of $S$ is the quotient division algebra of a quantum torus. As we observed in Remark~\ref{richardrmk}, the quotient division algebra of a quantum torus cannot, in characteristic $0$,  contain a copy of $A_1$.

\end{example}
\subsection{The commutative case}
Here we consider the case where $q_{ij}=1$ for all $i,j$ and $\mathcal{A}=\mathcal{O}_Q(\K^n)$ is the commutative polynomial algebra $\K[x_1,\dots,x_n]$. If $\sigma$ is inner on $\mathcal{T}$ then $\sigma=\id$ and $\Der_\sigma(\mathcal{A})=\Der(\mathcal{A})$ which, as  is well-known or can be deduced from
Proposition~\ref{sigmaderR2}(ii), is generated as an $\mathcal{A}$-module by the derivations $\partial/\partial x_j$. Here we determine for which non-trivial toric automorphisms $\sigma_\Lambda$ there are non-inner $\sigma_\Lambda$-derivations of $\mathcal{A}$.
\begin{prop}\label{comm}
Let $\mathcal{A}$ be the commutative polynomial algebra $\K[x_1,\dots,x_n]$, let $\sigma=\sigma_\Lambda$ be a non-trivial toric automorphism of $\mathcal{A}$, and let $\db\in \Z^n$ be such that  
$\Der_{\sigma,\db}(\mathcal{A})\neq \InnDer_{\sigma,\db}(\mathcal{A})$. 
\begin{enumerate}
\item There exists a unique $j$, $1\leq j\leq n$, such that $\lambda_j\neq 1$. 
\item For each $j$-exceptional $\db\in \Z^n$, $\Der_{\sigma,\db}(\mathcal{A})$ is one-dimensional spanned by the $x_j$-locally inner derivation induced by $\x^\db$.
\item Let $f\in \mathcal{A}$ with degree $0$ in $x_j$. There is an $x_j$-locally inner  $\sigma$-derivation $\delta_f$ of $\mathcal{A}$, induced by $(1-q)^{-1}x_j^{-1}f$, such that 
$\delta_f(x_j)=f$ and $\delta_f(x_i)=0$ for $i\neq j$. Every $\sigma$-derivation of $\mathcal{A}$ is the sum of an inner $\sigma$-derivation and $\delta_f$, for some $f\in \K[x_1,\dots,x_{j-1},x_{j+1},\dots,x_n]$.\end{enumerate}
\end{prop}
\begin{proof}
As $\mathcal{A}=\mathcal{O}_Q(\K^n)$ where $Q$ is the $n\times n$ matrix with each entry $q_{ij}=1$, $q_i(\db)=1$. As $\mathcal{T}$ is commutative, the only inner automorphism of $\mathcal{T}$ is $\id_\mathcal{T}$, so $\sigma$ is outer on $\mathcal{T}$ and there exists $j$, $1\leq j\leq n$, such that $\lambda_j\neq 1$.

(i) Suppose that $\db\in \N_0^n$. By Lemma~\ref{sigmaderR1}, $\Der_{\sigma,\db}(\mathcal{A})=\Der_{\sigma,\db}(\mathcal{T})$ so, by Proposition~\ref{dichotomy}(i),  $\Der_{\sigma,\db}(\mathcal{A})$ is one-dimensional spanned by the inner $\sigma$-derivation induced by $\x^\db$, contradicting the hypothesis that $\Der_{\sigma,\db}(\mathcal{A})\neq \InnDer_{\sigma,\db}(\mathcal{A})$. Hence $\db\notin \N_0^n$ and, by  Lemma~\ref{sigmaderR1}(ii), $\db$ is $j$-exceptional for some $j$, $1\leq j\leq n$.  By 
Proposition~\ref{sigmaderR2}(iii), $\lambda_i=1$ when $i\neq j$ so, as $\sigma$ is non-trivial, $\lambda_j\neq 1$.

(ii) As $q_i(\db)=1=\lambda_i$ when $i\neq j$ and $q_j(\db)=1\neq \lambda_j$, this is immediate from Proposition~\ref{sigmaderR2}(ii).

(iii) is immediate from (ii).
\end{proof}

\begin{rmk}
 In Proposition~\ref{comm}(iii), the Ore extension $\mathcal{A}[z;\sigma,\delta_f]$ is such that $zx_j=qx_jz+f$, $x_i$ is central for $i \neq j$, and $x_jz-(1-q)^{-1}f$ is normal. An example, with $n=1$ and $f=1-q$, is the quantum disc in Example \ref{qdisc}.
\end{rmk}

\section{Iterated Ore extensions of $\K$}\label{itOresection}

Throughout this section, let $n$ be a positive integer and let 
$R_1\subset R_2\subset\dots\subset R_n$ 
be a sequence of Ore extensions such that $R_1=\K[x_1]$ and, for $2\leq i\leq n$,
$R_{i}=R_{i-1}[x_i;\sigma_i,\delta_i]$ for some automorphism $\sigma_{i}$ of $R_{i-1}$ and
 some $\sigma_{i}$-derivation $\delta_i$ of $R_{i-1}$. Suppose that,
for $1\leq i<j\leq n$,
there exists $\lambda_{ij}\in \K^*$ such that $\sigma_{j}(x_i)=\lambda_{ij}x_i$.
The algebras $R_n$ include the quantum nilpotent  algebras \cite{GoodYakmem}, also known as $CGL$-extensions \cite{llrufd}. The definition of these has extra conditions 
involving rational actions, local nilpotency of the skew derivations and avoidance of roots of unity for the parameters, see \cite{GoodYakmem}. 

If $R_n$ is a quantum nilpotent algebra, there are two algorithms, due to Cauchon \cite{cauchon1} and to
Goodearl and Yakimov \cite{GoodYakadv}, that will embed $R_n$ in a quantum torus contained in the quotient division algebra $\Fract(R_n)$.
The results of Sections~\ref{sderqspace} and \ref{sdersellocqspace}, and their proofs, provide an algorithm which, for $R_n$ in general, will either embed $R_n$ in 
a selectively localized quantum space (and hence in a quantum torus) contained in  $\Fract(R_n)$ or show that $\Fract (R_n)$ contains a copy of  the first Weyl algebra $A_1$. 
If $\ch \K=0$ then, by Remark~\ref{richardrmk}, the algorithm will decide whether $R_n$ embeds in a quantum torus within $\Fract(R_n)$ and, if so, will produce such an embedding. 

An outcome that $R_n$ embeds in a quantum torus within $\Fract(R_n)$ arises through a sequence of selectively localized quantum spaces 
$S_1\subset S_2\subset\dots\subset S_n$ such that $S_1=R_1=\K[x_1]$ and, for $2\leq i \leq n$, (i) $R_i\subseteq S_i\subset \Fract(R_i)$, (ii) $\sigma_i$ and $\delta_i$ extend to $S_{i-1}$ and (iii) $S_{i}=S_{i-1}^\prime[yx_{i}-t_i;\sigma_i]$ where $t_i\in S_{i-1}$ and $S_{i-1}^\prime$ is a selectively localized quantum space obtained from $S_{i-1}$ by localization at the powers of some product 
$y$ of the canonical generators of $S_{i-1}$. 

Alternatively there may be a sequence as above, but only as far as $S_{j-1}$ for some $2\leq j\leq n$,  such that, on $S_{j}$, $\delta_{j}$ is 
$y$-locally conjugate to a non-zero derivation, for some product $y$ as before, in which case $\sigma_j$ is inner on the quantum torus
obtained  from $S_{j-1}$ on the  inversion of those canonical generators of $S_{j-1}$ that are not already units in $S_{j-1}$. In this case there is, by Corollary~\ref{quodivA1}, a copy of $A_1$ embedded in $\Fract(R_j)$ and hence in  $\Fract(R_n)$. 

The following lemma ensures that, after the replacement of $x_i$ by $yx_i-t_i$ as above, the condition that the generators of $R_i$ are eigenvectors for $\sigma_{j}$ for $i<j \leq n$ is inherited by the generators of $S_i$.
\begin{lemma}
Let $i\geq 2$, let $\ell=i-1$ and let $S=\mathcal{O}_{Q}(\K^\ell)_{\langle 1,\dots,k\rangle}$ be a selectively localized 
quantum space  for some multiplicatively antisymmetric $\ell \times \ell$ matrix $Q$. 
Let $\sigma$ be a toric automorphism of $S$  and let $\delta$ be a non-zero $x$-locally inner derivation, induced by $x^{-1}t$ for some product $x=x_{j_1}x_{j_2}\dots x_{j_m}$, where $1\leq m\leq n$ and $k+1\leq j_1<j_2<\dots j_m\leq \ell$, and some $t\in S$. 
Let $T=S[x_i;\sigma,\delta]$ and let $\tau$ be an automorphism of any overalgebra of $T$ such that 
$x_j$ is an eigenvector of $\tau$ for $1\leq j\leq i$. 
Then $xx_i-t$ is an eigenvector of $\tau$. 
\end{lemma}   
\begin{proof} For $\rho\in \K$, let $E_\tau(\rho)$ denote the eigenspace $\{v\in T:\tau(t)=\rho t\}$ for $\rho$, of the restriction of $\tau$ to $T$. There exist $\rho_1,\dots,\rho_{i-1},\rho_{i},\rho\in \K^*$ such that $x_j\in E_\tau(\rho_j)$ for $1\leq j\leq i$ and $x\in E_\tau(\rho)$. Note that 
$xx_i\in E_\tau(\rho\rho_i)$.

Let $\delta_1,\dots,\delta_m$ be the homogeneous components of $\delta$. These are the non-zero $x$-locally inner $\sigma$-derivations of $S$ induced by $x^{-1}t_1,\dots,x^{-1}t_m$, where $t_1,\dots,t_m$ are the homogeneous components  in $S$ of $t$. 
Let $1\leq k\leq m$. There exists $j$ such that $1\leq j\leq \ell$ and 
$\delta_k(x_j)\neq 0$. Then \[x_{i}x_j-\sigma(x_j)x_{i}=\delta(x_j)\] so $\delta(x_j)\in E_\tau(\rho_j\rho_i)$,
whence, as  $\delta_k(x_j)$ is a homogeneous component of $\delta(x_j)$, $\delta_k(x_j)\in E_\tau(\rho_j\rho_i)$. As $\delta_k(x_j)\in \K^* x^{-1}x_jt_k$, $x^{-1}x_jt_k\in E_\tau(\rho_j\rho_i)$. 
Hence $t_k\in E_\tau(\rho\rho_i)$ for all $k$ and so
$t\in E_\tau(\rho\rho_i)$ and $xx_{i}-t\in E_\tau(\rho\rho_i)$. 
\end{proof} 

\begin{algorithm}\label{Talg} The algorithm proceeds in up to $n-1$ stages, with Stage $i$ dealing with the adjunction of $x_i$. 
As $S_1=\K[x_1]$ the algorithm effectively begins with Stage 2. For $i\geq 2$, Stage $i$ begins with
a localization $S_{i-1}$ of $R_{i-1}$ that, for some multiplicatively antisymmetric $(i-1)\times(i-1)$ matrix $Q_{i-1}$, is a selectively localized quantum space $\mathcal{O}_{Q_{i-1}}(\K^{i-1})_{\langle 1,2,\dots,k\rangle}$ together with
a toric automorphism $\sigma$ of $S_{i-1}$ and a $\sigma$-derivation $\delta$, that are extensions from $R_{i-1}$ of $\sigma_i$ and $\delta_i$ respectively.  For the duration of Stage $i$, we shall renumber and rename the canonical generators of $S_{i-1}$ as $y_1^{\pm 1}$, $y_2^{\pm 1},\dots y_k^{\pm 1}, y_{k+1},\dots y_{i-1}$. 

Stage $i$ involves a number of steps. 
{\it Step(a)} is the computation of $\sigma(y_j)$ for those $y_j$'s that were not among the original generators of $R_{i-1}$  and the identification of $\Lambda\in (\K^*)^{i-1}$
such that, on $S_{i-1}$, $\sigma=\sigma_\Lambda$. 
If $\delta=0$ then the output of Stage $i$ is $S_i=S_{i-1}[x_i;\sigma]$, which is the
selectively localized quantum space $\mathcal{O}_{Q_i}(\K^i)_{\langle 1,2,...,k\rangle}$, where $Q_i$ is the multiplicatively antisymmetric $i\times i$ matrix obtained from $Q_{i-1}$ by adding an extra row and column so that $q_{ij}=\lambda_j$ for $1\leq j<i$. 

\textit{Step (b)} implements Proposition~\ref{sigmaderR2}/\ref{sigmaderRk2} and is to be followed for each $j$, $1\leq j\leq i-1$. It begins with the computation of
$\delta(y_j)$ in terms of $y_1,y_2,\dots,y_{i-1}$, rather than in terms of $x_1,x_2,\dots,x_{i-1}$, and the identification of those $\db\in \Z^n$ such that $\delta_{\db}(y_j)\neq 0$ and $\delta_{\db}(y_i)=0$ for $i<j$.  For each such $\db$, either  $\db$ is $j$-exceptional or $\db\in \Z^k\times \N_0^{n-k}$ and  
 we proceed as follows to determine the nature of the component $\delta_{\db}$.

If $\db$ is $j$-exceptional then
compute $q_j(\db)$. If $q_j(\db)=\lambda_j$ then
the algorithm terminates with the outcome that $\Fract(R_n)$ contains a copy of $A_1$, which may be identified following the proof of Corollary~\ref{quodivA1}.  If $q_j(\db)\neq \lambda_j$ then, comparing $\delta(y_j)$ and $\psi(y_j)$ where $\psi$ is  $y_j^{-1}$-locally inner, induced by $y_j^{-1}\y^{\db+\eb_j}$,
identify $u \in \K^*$ such that $\delta_{\db}=u\psi$. 

If $\db$ is not $j$-exceptional compute $q_\ell(\db)$ for 
all $\ell$, $1\leq \ell\leq i-1$. If $q_\ell(\db)=\lambda_\ell$ for all $\ell$ then 
the algorithm terminates with the outcome that $\Fract(R_n)$ contains a copy of $A_1$. If $q_\ell(\db)\neq \lambda_\ell$ for some $\ell$ then 
identify $v \in \K^*$ such that $\delta_{\db}=v\xi$ where $\xi$ is inner, induced by $\y^\db$.

{\it Step(c)} implements Proposition~\ref{mixed}/\ref{mixedsel} and is only reached if every non-zero homogeneous component of $\delta$ is either $y_j$-locally inner for some $j$ or inner. Taking common denominators, identify $t\in S_{i-1}$ such that $\delta$ is $y$-locally inner, 
induced by $y^{-1}t$, where $y=y_{j_1}y_{j_2}\dots y_{j_m}$, $j_1,j_2,\dots,j_m$ being the distinct values of $j>k$ for which $\delta$ has 
a non-zero $y_j$-homogeneous component of  $j$-exceptional weight. 

Let $y_i$ be the normal element $yx_{i}-t$ of $S_{i-1}[x_i;\sigma,\delta]$. Identify the multiplicatively antisymmetric $i\times i$ matrix $Q^\prime$ such that  the localization $S_{i}$ of $S_{i-1}[x_i;\sigma,\delta]$ at the powers of $y$ is the selectively localized 
 quantum space $\mathcal{O}_{Q_i}(\K^{m+1})_{\langle i_1,\dots,i_{k+\ell}\rangle}$ with generators $y_1^{\pm 1}$, $y_2^{\pm 1},\dots y_k^{\pm 1}, y_{k+1},\dots y_{i-1},y_{j_1}^{-1},y_{j_2}^{-1},\dots,y_{j_m}^{-1},y_i$.
\end{algorithm}

\begin{example}
Here we  follow the algorithm in a well-known example  chosen to illustrate the algorithm itself  rather than to add to understanding of the example. Let $R_9$ be the single parameter algebra of $3\times3$ quantum matrices, with the generators $x_{ij}$, $1\leq i,j\leq 3$, adjoined in the order $x_{11}, x_{12},x_{21},x_{22},x_{13},x_{23},x_{31},x_{32},x_{33}$, so that $R_4$ is the algebra of $2\times 2$ quantum matrices, as in Example~\ref{2by2qmat}, and $R_6$ is the algebra of $2\times 3$ quantum matrices. The parameter $q\in \K^*$ is such that $q^2\neq 1$. 
For $4\leq i\leq 9$, $Q_i$ will denote the $i \times i$ submatrix of the matrix 
\[
\left(\begin{array}{ccccccccc}
1 &q &q&1&q&q&q&q&1\\
q^{-1}& 1 &1&1&q&1&1&q&1\\
q^{-1}& 1 &1&1&1&q&q&1&1\\
1&1&1&1&q&q&q&q&1\\
q^{-1}&q^{-1}&1&q^{-1}&1&1&1&1&1\\
q^{-1}&1&q^{-1}&q^{-1}&1&1&1&q&1\\
 q^{-1}& 1&q^{-1}&q^{-1} & 1& 1&1&q&1\\
q^{-1}&q^{-1}&1& q^{-1}&1& q^{-1}& q^{-1}&1&1\\
1&1&1&1&1&1&1&1&1
\end{array}\right)\]
by deleting rows $i+1$ to $9$ and columns $i+1$ to $9$.
 
We shall not present full details of the calculations which are routine. 
Stages 2-4 in the algorithm for $R_9$ were, in effect,  carried out in Example~\ref{2by2qmat}, with $p=q$ and $r=1$. At Stages 2 and 3, the derivation $\delta$ is $0$ and at Stage 4, $\delta$ is $x_{11}$-locally inner, induced by $qx_{11}^{-1}x_{12}x_{21}$. The output from Stage 4 is the selectively localized quantum space 
$S_4:=\mathcal{O}_{Q_4}(\K^4)_{\langle 1\rangle}$, with canonical generators $x_{11}^{\pm 1}$, $x_{12}$, $x_{21}$ and the quantum determinant 
$y_{22}=x_{11}x_{22}-qx_{12}x_{21}$. 

At Stage 5, $R_5=R_4[x_{13};\sigma]$, where  $\sigma(x_{11})=q^{-1}x_{11}$, $\sigma(x_{12})=q^{-1}x_{12}$,
$\sigma(x_{21})=x_{21}$ and $\sigma(x_{22})=x_{22}$.

At  {Step(a)}, $\sigma(y_{22})=q^{-1}y_{22}$ so $\sigma=\sigma_\Lambda$, where $\Lambda=(q^{-1},q^{-1},1,q^{-1})$. The skew derivation $\delta_5$ is $0$ so
the output from Stage 5 is
the selectively localized quantum space 
$S_5:=\mathcal{O}_{Q_5}(\K^5)_{\langle 1\rangle}$, with canonical generators $x_{11}^{\pm 1}$, $x_{12}$, $x_{21}$, $y_{22}$ 
and $x_{13}$.

At Stage 6, $R_6=R_5[x_{23};\sigma,\delta]$ where $\sigma(x_{11}) = x_{11}$, $\sigma(x_{12})=x_{12}$,
$\sigma(x_{21}) = q^{-1}x_{21}$, $\sigma(x_{22}) = q^{-1}x_{22}$, $\sigma(x_{13}) = q^{-1}x_{13}$, $\delta(x_{11}) = (q^{-1}-q)x_{13}x_{21}$,
$\delta(x_{12}) = (q^{-1}-q)x_{13}x_{22}$   and $\delta(x_{21}) = \delta(x_{22}) =
\delta(x_{13}) = 0$. This is the algebra $\mathcal{O}_q(M_{2,3}(\K))$ of quantum
$2\times 3$ matrices.

At Step(a),
$\sigma(y_{22}) = q^{-1}y_{22}$ and, on $S_5$, $\sigma=\sigma_\Lambda$, where 
$\Lambda=(1,1,q^{-1},q^{-1},q^{-1})$. 

At Step(b), with $\ell=1$, $\delta(x_{11})=(q^{-1}-q)x_{13}x_{21}$, so there is a
homogeneous component $\delta_1$, which has degree $\db_1=(-1,0,1,0,1)$.
Here $\lambda_1=1\neq q^{-2}=q_1(\db_1)$ and 
$\delta_1$ is inner on  $S_5$, induced by $qx_{11}^{-1}x_{13}x_{21}$. 

With $\ell=2$,  $\delta(x_{12})$ has a term of degree $(-1,1,1,0,1)$, which must be $\delta_1(x_{12})$, and  
$\delta(x_{12})=\delta_1(x_{12})+\delta_2(x_{12})$, where $\delta_2$ is homogeneous of $2$-exceptional degree $\db_2=(-1,-1,0,1,1)$ and 
$\delta_2(x_{12})=(q^{-1}-q)x_{11}^{-1}y_{22}x_{13}$. 
As $\lambda_2=1\neq q^{-2}=q_2(\db_2)$, 
$\delta_2$ is $x_{12}$-locally inner on  $S_5$, induced by $qx_{12}^{-1}x_{11}^{-1}
y_{22}x_{13}$.

At Step(c), 
$\delta$ is $x_{12}$-locally inner on $S_5$, induced by $x_{12}^{-1}
t$, where $t=qx_{11}^{-1}(y_{22}+qx_{12}x_{21})x_{13}=qx_{22}x_{13}$. 
Let
$y_{23}= x_{12}x_{23}-t$ which, in terms of the original generators of $R_6$,  is the \emph{quantum minor} $x_{12}x_{23}-qx_{22}x_{13}\in R_6$.
The output from Stage 6 is the selectively localized quantum space 
$S_6:=\mathcal{O}_{Q_6}(\K^6)_{\langle 1,2\rangle}$, with canonical generators $x_{11}^{\pm 1}$, $x_{12}^{\pm 1}$, $x_{21}$, $y_{22}$, $x_{13}$ and $y_{23}$.

At Stage 7, 
$R_7=R_6[x_{31};\sigma]$, where 
$\sigma(x_{11})=q^{-1}x_{11}$, 
$\sigma(x_{12})=x_{12}$, $\sigma(x_{21})=q^{-1}x_{21}$, $\sigma(x_{22})=x_{22}$, $\sigma(x_{13}) = x_{13}$ and $\sigma(x_{23})=x_{23}$.
At Step(a), $\sigma(y_{22})=q^{-1} y_{22}$ and $\sigma(y_{23})=y_{23}$ so, on $S_6$,
$\sigma=\sigma_\Lambda$ where 
$\Lambda=(q^{-1},1,q^{-1},q^{-1},1,1)$. 
 The skew derivation being $0$,
the output from Stage 7 is 
the selectively localized quantum space 
$S_7:=\mathcal{O}_{Q_7}(\K^7)_{\langle 1,2\rangle}$, with canonical generators $x_{11}^{\pm 1}$, $x_{12}^{\pm 1}$, $x_{21}$, $y_{22}$, $x_{13}$, $y_{23}$ and $x_{31}$.

At Stage 8, $R_8=R_7[x_{32};\sigma,\delta]$ where 
$\sigma(x_{11}) = x_{11}$, $\sigma(x_{12})=q^{-1}x_{12}$,
$\sigma(x_{21}) = x_{21}$, $\sigma(x_{22}) = q^{-1}x_{22}$, $\sigma(x_{13}) = x_{13}$, $\sigma(x_{23}) = x_{23}$, $\sigma(x_{31}) = q^{-1}x_{13}$
while $\delta(x_{11}) = (q^{-1}-q)x_{12}x_{31}$,
$\delta(x_{21}) = (q^{-1}-q)x_{22}x_{31}$,  and $\delta(x_{12}) = \delta(x_{13}) =
\delta(x_{22}) =\delta(x_{23})=\delta(x_{31})=0$.
The calculations at this stage can be expedited by exploiting the known automorphism $\tau$  of $\mathcal{O}_q(M_{3}(\K))$ such that $\tau(x_{ij}) = x_{ji}$ for $1 \leq i, j \leq 3$ (see \cite{LaunLen}), and the corresponding calculations in the adjunction of $x_{23}$ at Stage 6.

At Step(a),
$\sigma(y_{22})=q^{-1}y_{22}$, and $\sigma(y_{23})=q^{-1}y_{23}$ and, as an automorphism of $S_7$, $\sigma=\sigma_\Lambda$ where 
$\Lambda=(1,q^{-1},1,q^{-1},1,q^{-1},q^{-1})$.

At 
Step(b) with $\ell=1$, yields a homogeneous component $\delta_1$, of degree $(-1,0,0,1,0,0,1)$, such that 
$\delta_1(x_{11})=(q^{-1}-q)x_{12}x_{31}$ and 
$\delta_1$ is  inner on $S_7$,
induced by $qx_{11}^{-1}x_{12}x_{31}$.

With $\ell=2$, $\delta(x_{21})=0$, yielding no new homogeneous component and, with $\ell=3$,
$\delta(x_{21})=\delta_1(x_{21})+ (q^{-1}-q)x_{11}^{-1}x_{12}x_{21}x_{31}$, there is an $x_{21}$-locally inner homogeneous component $\delta_2$, of degree $d_2=(-1,0,-1,1,0,0,1)$ on $S_7$, induced
by $qx_{21}^{-1} x_{11}^{-1} y_{22}x_{31}$.

As
$\delta(y_{22})=0=\delta(y_{13})=\delta(y_{23})=\delta(y_{31})$, there are no further homogeneous components of $\delta$.

At Step(c),
$\delta$ is $x_{21}$-locally inner on $S_7$, induced by $x_{21}^{-1}
t$, where $t=qx_{11}^{-1}(y_{22}+x_{12}x_{21})x_{31}=qx_{22}x_{31}$. Let
$y_{32}= x_{21}x_{32}-t$. In terms of the original generators of $R_8$, $y_{32}$ is the \emph{quantum minor} $x_{21}x_{32}-qx_{22}x_{31}$.
The output from Stage 8 is the selectively localized quantum space 
$S_8:=\mathcal{O}_{Q_8}(\K^8)_{\langle 1,2,3\rangle}$, with canonical generators $x_{11}^{\pm 1}$, $x_{12}^{\pm 1}$, $x_{21}^{\pm 1}$, $y_{22}$, $x_{13}$, $y_{23}$, and $y_{32}$

At Step(a) of Stage 9, $\sigma(y_{22}) = y_{22}$, $\sigma(y_{23})=q^{-1}y_{23}$ and $\sigma(y_{32}) = q^{-1}y_{32}$ so, on $S_8$, $\sigma=\sigma_\Lambda$, where $\Lambda=(1,1,1,1,q^{-1},q^{-1},q^{-1},q^{-1})$.
At Step(b) with $\ell=1$, there is an inner homogeneous component $\delta_1$ on $S_8$, of degree $(-1,0,0,0,1,0,1,0,0)$, induced by $qx_{11}^{-1}x_{13}x_{31}$, such that $\delta(x_{11})=(q^{-1}-q)x_{13}x_{31}$.

With $\ell=2$,
$\delta(x_{12})=\delta_1(x_{12})+(q^{-1}-q)x_{21}^{-1}x_{13}y_{32}+(q^{-1}-q)x_{11}^{-1}x_{21}^{-1}y_{22}x_{13}x_{31}$ and  there are homogeneous components $\delta_2$, of degree $(0,-1,-1,0,1,0,0,1)$, and $\delta_3$, of degree $(-1,-1,-1,1,1,0,1,0)$.
Here
$\delta_2$ and $\delta_3$ are inner on $S_8$,
induced by  $qx_{12}^{-1}x_{21}^{-1}x_{13}y_{32}$ and $x_{11}^{-1}x_{12}^{-1}x_{21}^{-1}y_{22}x_{13}x_{31}$ respectively.

With $\ell=3$,
there is a homogeneous component $\delta_4$, of degree $(0,-1,-1, 0, 0, 1, 1, 1)$, such that 
$\delta_4(x_{21})=(q^{-1}-q)x_{12}^{-1}y_{23}x_{31}$ and 
$\delta(x_{21})=(\delta_1+\delta_2+\delta_3+\delta_4)(x_{21})$. Here
$\delta_4$ is inner on $S_8$,
induced by $x_{11}^{-1}x_{12}^{-1}x_{21}^{-1}y_{22}x_{13}x_{31}$.

With $\ell=4$,
a substantial calculation shows that there is a $y_{22}$-locally inner homogeneous component $\delta_5$ on $S_8$, of degree  $(1,-1,-1,-1,0,1,0,1)$, induced by
$q^2x_{11}x_{12}^{-1}x_{21}^{-1}y_{22}^{-1}y_{23}y_{32}$ and such that $\delta(y_{22})=(\delta_1+\delta_2+\delta_3+\delta_4+\delta_5)(y_{22})$ and $\delta_5(y_{22})=(1-q^2)x_{11}x_{12}^{-1}x_{21}^{-1}y_{23}y_{32}$,

As $\delta(x_{13})=\delta(y_{23})=\delta(x_{31})=\delta(y_{32})=0$, there are no further homogeneous components of $\delta$.

At Step(c),
$\delta=\delta_1+\delta_2+\delta_3+\delta_4+\delta_5$ is $y_{22}$-locally inner on $S_8$, induced by $y_{22}^{-1}
t$, where  
\begin{multline*}
t = qx_{11}^{-1}y_{22}x_{13}x_{31}
+qx_{12}^{-1}x_{21}^{-1}y_{22}^2x_{13}y_{32}
+x_{11}^{-1} x_{12}^{-1}x_{21}^{-1}y_{22}^2x_{13}x_{31}\\
+qx_{12}^{-1} x_{21}^{-1}y_{22}^2y_{23}x_{31}
+q^2x_{11}x_{12}^{-1}x_{21}^{-1}y_{23}y_{32}.
\end{multline*}
Let
$y_{33}=y_{22}x_{33}-t$. The output from Stage 9 is the selectively localized quantum space 
$\mathcal{O}_{Q_9}(\K^9)_{\langle 1,2,3,4\rangle}$,
with canonical generators $x_{11}^{\pm 1}$, $x_{12}^{\pm 1}$, $x_{21}^{\pm 1}$, $y_{22}^{\pm 1}$, $x_{13}$, $y_{23}$, $x_{31}$, $y_{32}$ and $y_{33}$.

It can be checked, not without
pain, that in terms of the original generators of $R_9$, $y_{33}$ is the quantum determinant 
of quantum $3\times 3$ matrices, defined as in \cite[I.2.3]{BGl},
which is known to be central, see, for example, \cite[Exercise I.2.E]{BGl}.

\end{example}

\end{document}